\documentclass[preprint,12pt]{article}
\usepackage{amsmath,amsthm,amssymb,amscd,fancybox,epsfig,float}
\usepackage[margin=1in]{geometry}
\usepackage{authblk}

\usepackage{hyperref}
\usepackage{graphicx}
\usepackage{bbm}
\usepackage{amsmath,amsthm,amssymb,amscd}
\usepackage{caption,subcaption}
\usepackage{booktabs,multirow}
\usepackage{setspace} 
\usepackage{color}

\usepackage{arydshln}
\usepackage{graphics}
\usepackage{epstopdf}
\usepackage{multirow}
\usepackage{hhline}
\usepackage{cleveref}
\usepackage{xcolor}
\usepackage{tikz}
\usetikzlibrary{arrows,automata}
\usepackage{cite}
\newcommand{\beqn}{\begin{equation}}
\newcommand{\eeqn}{\end{equation}}

\newcommand{\bbmat}{\begin{bmatrix}}
\newcommand{\ebmat}{\end{bmatrix}}

\newtheorem{definition}{Definition}[section]
\newtheorem{remark}{Remark}[section]
\newtheorem{remark1}{Remark}[section]

\newtheorem{corollary}{Corollary}[section]
\newtheorem{theorem}{Theorem}[section]
\newtheorem{proposition}{Proposition}[section]
\newtheorem{lemma}{Lemma}[section]

\title{On the inverse scattering problem for radially-symmetric domains in two dimensions}
\author{
Abinand Gopal\thanks{
Department of Mathematics, Yale University, New Haven, CT 06511.
email: abinand.gopal@yale.edu} ,\,
Jeremy Hoskins\thanks{
Department of Statistics, University of Chicago, Chicago, IL 60637.
email: jeremyhoskins@uchicago.edu} ,\, 
Vladimir Rokhlin\thanks{Department of Mathematics and Department of Computer Science, Yale University, New Haven, CT, 06511. 
email: rokhlin@cs.yale.edu}}

\date{}

\begin{document}

\maketitle
%%%%%%%%%%%%%%%%%%%%%%%%%%%%%%%%%%%%%%%%%%%%%%%%%%%%%%%%%%%%
\abstract{
   We present a new procedure for solving radially-symmetric, acoustic
   inverse scattering problems in the plane, given multifrequency data.
   Our approach builds upon the previous work \cite{yuchen}, which solved the 
   one-dimensional problem using a trace formula. 
   In particular, we develop a new trace formula relating the impedance
   of the field to the scattering potential, in the setting where the
   scattering potential is radially symmetric.
   We show that the resulting integro-differential equation can be solved in
   a stable and high-order manner, yielding a viable numerical procedure for
   solving the inverse problem. 
   We demonstrate the efficacy of our approach using several numerical
   experiments. 
%
%    ... Introduction
%

\section{Introduction}  \label{30}
Acoustic inverse scattering problems arise in a variety of applications such
as geophysics, medical imaging, nondestructive testing, and sonar. 
The basic problem is to recover the scattering potential from measurements
of the scattered field.
This is a highly nonlinear problem with all of the associated challenges. 

In the one-dimensional setting, it is often possible to
transform the wave equation to the Schr\"odinger equation for which
relatively straightforward techniques exist (see
\cite{chad1,gelf}).
These mappings may not be numerically stable, however, and it is unknown how
to produce analogous mappings in two or three dimensions.
Alternatively, the nonlinear problem can be linearized using, for example,
the Born approximation (see \cite{stark}).
This sometimes works in the low-contrast regime, but breaks down in
environments with large amounts of backscattering.

Another approach is to solve the fully nonlinear problem.
This can be done using nonlinear optimization by solving a sequence of
linearized problems (see, for example,
\cite{symes1,lines1,pan1,lgc4,lgc5,lgc10,lgc11,lgc49,lgc44,lgc45,lgc46,lgc58,lgc59,les1}).
Another possibility is to use techniques from signal and image processing to
directly solve the nonlinear equation (see, for example,
\cite{lgc31,lgc35,lgc68,lgc51}).

For one-dimensional problems, yet another possibility is to use so-called
\emph{trace formulas} which relate multifrequency field data to the
scattering potential (see, for example, \cite{deift1,crutch1,stick1,sylv1}). 
In \cite{yuchen}, a stable and computationally efficient numerical
procedure was presented for solving one-dimensional inverse scattering
problems based on a trace formula.
In this manuscript, we extend the approach of \cite{yuchen} to
two-dimensional problems under the assumption that the scattering potential
is radially symmetric. 

We now outline the remainder of this paper. 
In Section 2, we formally state the problem and state some existing results
that are used in this work.
This is followed by Section 3 which contains our mathematical apparatus.
In Section 4, we state and prove our principal result, which is a new trace
formula for the two-dimensional, radially-symmetric case.
Our numerical procedure, along with several numerical experiments, are
presented in Section 5. 
Finally, in Section 6, we summarize our results and briefly outline some
directions for future work.

\section{Preliminaries}
\subsection{Formulation of the problem}
In this paper we consider the inverse scattering problem for acoustic waves
in radially-symmetric annuli in two dimensions. Let $\Omega$ denote the
annulus centered at the origin with inner radius $a$ and outer radius $b.$
Let $B_a$ denote the ball of radius $a$ centered at the origin. At a single
frequency $k \in \mathbb{C}$ such that ${\rm Im}\,k \ge 0,$ the
time-harmonic acoustic wave equation for the scattered field is
\begin{align}\label{eqn:helmholtz}
\Delta u({\bf x}) + k^2 \left(1+Q({\bf x})\right) u({\bf x}) &= f({\bf x})
\end{align}
subject to the Sommerfeld radiation condition
\begin{align}
\lim_{r\rightarrow \infty} \sqrt{r}\left(\frac{\partial u}{\partial r}-ik u \right) &= 0.\nonumber
\end{align}
Here we assume the source $f\in L^2$ is a function supported on $B_a$ and the potential $Q$ is a continuous compactly-supported radially-symmetric function. In particular, we assume that $Q({\bf r}) = q(\|{\bf r}\|)$ for some continuous function $q:\mathbb{R}\rightarrow \mathbb{R}$  supported on the interval $[a,b]$ with $0<a<b<\infty.$ Here $\| \cdot\|$ denotes the standard Euclidean norm in $\mathbb{R}^2.$ Moreover, we assume that there exist two constants $q_0$ and $q_1$ such that $-1<q_0<q(r)<q_1<\infty$ for all $a<r<b.$

\subsection{Reduction to the radial problem}
In this section we reduce the radially-symmetric acoustic scattering problem
to a set of decoupled one-dimensional scattering problems. Let $u_n:
[0,\infty) \rightarrow \mathbb{C}$ be the Fourier coefficient of $u$ with
respect to the angle $\theta,$ i.e.
\begin{align}\label{eqn:un_def}
u_n(r) = \int_0^{2\pi} e^{-in\theta}u(r \cos\theta,r \sin\theta)\,{\rm d}\theta.
\end{align}
For any integer $n$ the function $u_n$ satisfies the differential equation
\begin{align}\label{eqn:1d}
u_n''(r) + \frac{1}{r}u_n'(r)+k^2[1+q(r)]u_n(r)-\frac{n^2}{r^2}u_n(r)=f_n(r),
\end{align}
where 
\begin{align}
f_n(r) = \int_0^{2\pi} e^{-in\theta}f(r \cos\theta,r \sin\theta)\,{\rm d}\theta
\end{align}
is the Fourier coefficient of $f.$

\begin{remark1}\label{lem_soln_char}
On any interval $c<x<d$ on which the source $f$ and potential $q$ are
identically zero the solutions to equation (\ref{eqn:1d}) are linear
combinations of the Bessel function $J_n(kr)$ and the Hankel function
$H_n(kr).$ Specifically,  if $f_n$ is supported on the interval $[0,R]$ and
$q$ is supported on the interval $[a,b]$ then for all $A \in \mathbb{C}$
there exist constants $\alpha,\mu$ depending only on the source $f_n$ and
potential $q$ such that for all $R<r<a$ 
\begin{align}
u_n(r) = A H_n(kr) + \alpha J_n(kr)
\end{align}
and for all $r>b,$ 
\begin{align}
u_n(r) = \mu H_n(kr).
\end{align}
\end{remark1}

\begin{remark1}\label{rem:scalemu}
The data required by the recovery algorithm presented in this paper depends only on the quantity $u_n'(r)/u_n(r).$ Hence in the remainder of the paper we will assume that solution $u_n(r)$ is scaled so that $\mu =1.$
\end{remark1}

In the following it will be convenient to rescale $u_n$ by $\sqrt{r}.$ This new quantity, $\sqrt{r}u_n(r),$ also satisfies a differential equation which can be readily obtained from equation (\ref{eqn:1d}).
\begin{lemma}
Let $u_n$ be a solution to the differential equation
\begin{align}\label{eqn:1dm}
u_n''(r) + \frac{1}{r}u_n'(r)+k^2[1+q(r)]u_n(r)-\frac{n^2}{r^2}u_n(r)=0,
\end{align}
and define the function $\psi_n:[a,b] \to \mathbb{C}$ by 
\begin{align}\label{eqn:psi_def}
\psi_n(r) = \sqrt{r} \,u_n(r).
\end{align}
 Then $\psi_n$ satisfies the equation
\begin{align}\label{eqn:psim_def}
\psi_n''(r) + k^2\left(1+q(r)\right) \psi_n(r)-\frac{n^2-\frac{1}{4}}{r^2} \psi_n = 0,
\end{align}
with the boundary conditions
\begin{align}
\psi_n(b) &= \sqrt{b}\, H_n(kb),\\
\psi_n'(b)&= k\sqrt{b}\, H_n'(kb)+\frac{1}{2b} H_n(kb).
\end{align}
\end{lemma}

\begin{remark}
The function $\psi_n:[a,b] \to \mathbb{C}$ can be extended to a differentiable function
defined on $(0,\infty).$ Specifically, for $r>b$ we set
\begin{align}
\psi_n(r) = \sqrt{r} H_n(kr)
\end{align} 
and for $r<a$ we set
\begin{align}
\psi_n(r) = \alpha \sqrt{r} H_n(kr)+ \beta \sqrt{r} J_n(kr)
\end{align} 
where the coefficients $\alpha$ and $\beta$ are chosen so that $\psi_n(r)$ and $\psi_n'(r)$ are continuous at $a.$
\end{remark}

\subsection{Impedance}
In this section we introduce the concept of impedance (see, for example, \cite{sylv1}) and summarize its properties which are relevant to the subsequent analysis.
\begin{definition}\label{def_imped}
Given a solution $\psi_n$ of (\ref{eqn:psi_def}) the impedance $\phi_n:[a,b] \times \{z \in \mathbb{C} : z \neq 0, \,{\rm Im}z \ge 0\} \to \mathbb{C}$ is the function defined by the formula
\begin{align}\label{def_phin}
\phi_n(r,k) = \frac{ \psi_n'(r)}{ik \psi_n(r)}.
\end{align}
\end{definition}

\begin{remark}
The impedance $\phi_n$ corresponds to a scattering problem in which the source is located inside an annulus and propagates outward to infinity. One could in principle define an inward impedance corresponding to an incoming wave impinging upon an annulus and reflecting outward. For the inward impedance interference from waves passing through the scatterer from opposite directions produces poles which necessitate a different approach. A detailed analysis of the trace formula for the inward impedance will be published at a later date. 
\end{remark}

%\begin{remark}
%We note that for all $r>b$ 
%\begin{align}
%\phi_n(r,k) = \frac{H_n'(kr)}{i H_n(kr)}+\frac{1}{2irk},
%\end{align}
%and hence is independent of the constant $\mu$ in Lemma \ref{lem_soln_char}. In particular, without loss of generality we may assume $\mu=1.$
%\end{remark}

The definition of the impedance and equation (\ref{eqn:psim_def}) immediately imply the following lemma.

\begin{lemma}\label{lem:ricat_eqn}
For all $a \le r \le b,$ the impedance $\phi_n$ satisfies the Riccati equation
\begin{align}
\frac{\partial}{\partial r}\phi_n(r,k) = -ik\phi_n^2(r,k)-ik (1+q(r)) +i\frac{n^2-\frac{1}{4}}{kr^2}, 
\end{align}
together with the boundary condition
\begin{align}
\phi_n(b,k) =  \frac{H_n'(kb)}{i H_n(kb)}+\frac{1}{2ibk}.
\end{align}
\end{lemma}

\begin{remark1}
In a mild abuse of notation unless otherwise stated we will denote derivatives of $\phi_n(r,k)$ with respect to $r$ by $\phi_n'(r,k).$ 
\end{remark1}

\begin{corollary}\label{cor:ricat_eqn}
Suppose $n$ is a non-negative integer and consider the function
$$w:[a,b] \times \{k \in \mathbb{C} : k \neq 0, \,{\rm Im}\,k \ge 0\} \to \mathbb{C}$$
 defined by
\begin{align}\label{eqn:wdeff}
w(r,k) = \phi_n(r,k)-\frac{(\sqrt{r}H_n(kr))'}{ik \sqrt{r} H_n(kr)},
\end{align}
where $\phi_n$ is the impedance defined in (\ref{def_phin}). Then for all $a \le r \le b$
\begin{align}\label{eqn:wdiff}
w'(r,k) =  -ik w(r,k) \left(w(r,k)+2\frac{(\sqrt{r}H_n(kr))'}{ik \sqrt{r} H_n(kr)}\right) -ik\,q(r),
\end{align}
and $w(b,k) = 0$ for all non-zero $k \in \mathbb{C}$ with non-negative imaginary part. Obviously, the differential equation (\ref{eqn:wdiff}) is equivalent to the integral equation
\begin{align}\label{eqn:w_intfrm}
w(r,k) =  ik \int_r^b \left(w(x,k) \left(w(x,k)+2\frac{(\sqrt{r}H_n(kx))'}{ik \sqrt{r} H_n(kx)}\right) +q(x)\right) \,{\rm d}x,\,\,\,\,\, a \le r \le b.
\end{align}
\end{corollary}

We conclude this section with the following lemma which characterizes the symmetry of the impedance in frequency.

\begin{lemma}\label{lem:phi_sym}
Let $k$ be a non-zero real number and $0<r<\infty.$ Then for all non-negative integers $n$
\begin{align}
\phi_n(r,k) = \overline{\phi_n(r,-k)}.
\end{align}
\end{lemma}

\subsection{Properties of Bessel functions}
In this section we list certain properties of Bessel and Hankel functions which will be used in the subsequent analysis.
\begin{proposition}
Let $n$ be a non-negative integer and $z$ be a non-zero complex number with a non-negative imaginary part. The Bessel functions of the first and second kind have the following expansions about $z=0$
\begin{align}\label{eqn:bess_exp}
J_n(z) = \left(\frac{z}{2} \right)^n \sum_{n=0}^\infty&(-1)^k \frac{\left(\frac{1}{4}z^2\right)^k}{k! \,\Gamma(n+k+1)},\nonumber\\
Y_n(z) = -\frac{\left(\frac{1}{2} z \right)^{-n}}{\pi} &\sum_{k=0}^{n-1}\frac{(n-k-1)!}{k!} \left(\frac{1}{4}z^2\right)^k+\frac{2}{\pi} \log\left(\frac{1}{2}z \right)\,J_n(z)-\\
&\frac{\left(\frac{1}{2}z\right)^n}{\pi}\sum_{k=0}^\infty \left(\Psi(k+1)+\Psi(n+k+1) \right) \frac{\left(-\frac{1}{4}z^2\right)^k}{k!(n+k)!},\nonumber
\end{align}
where $\Psi(z) = \Gamma'(z)/\Gamma(z)$	and we take the branch cut of $\log$ to lie along the negative imaginary axis.

Moreover, if $H_n(z)$ denotes the $n$th order Hankel function then
\begin{align}\label{eqn:hank_exp}
H_n(z) &= \left(\frac{z}{2} \right)^n \sum_{n=0}^\infty(-1)^k \frac{\left(\frac{1}{4}z^2\right)^k}{k! \,\Gamma(n+k+1)}\\
& -i\frac{\left(\frac{1}{2} z \right)^{-n}}{\pi} \sum_{k=0}^{n-1}\frac{(n-k-1)!}{k!} \left(\frac{1}{4}z^2\right)^k+i\frac{2}{\pi} \log\left(\frac{1}{2}z \right)\,J_n(z)-\nonumber\\
&i\frac{\left(\frac{1}{2}z\right)^n}{\pi}\sum_{k=0}^\infty \left(\Psi(k+1)+\Psi(n+k+1) \right) \frac{\left(-\frac{1}{4}z^2\right)^k}{k!(n+k)!}.\nonumber
\end{align}
\end{proposition}
\begin{remark}\label{rem:hnearzero}
It follows immediately from (\ref{eqn:hank_exp}) that for all $n=0,1,2,\dots$ there exists a constant $C_n$ depending only on $n$ such that if $|z| \le C_n$ then
\begin{align}\label{eqn:hnearzero}
\left| \frac{H_n'(z)}{H_n(z)} \right| \le \frac{4(n+1)}{|z|}
\end{align}
\end{remark}

Hankel functions also have the following asymptotic expansions valid for large arguments.
\begin{proposition}
Let $n$ be a non-negative integer and $z \in \mathbb{C}$ such that $\Im (z) \ge 0.$ Then
\begin{align}
H_n(z) \sim \frac{2}{\pi i^{n+1}} \sqrt{\frac{\pi}{2z}}e^{iz+\frac{i\pi}{4}}\left(1-\frac{n^2-1}{-8 iz}+\frac{(4n^2-1)(4n^2-9)}{2(-8iz)^2}+\dots\right)
\end{align}
\end{proposition}

\begin{corollary}\label{cor:hank_wkb}
Let $n$ be a non-negative integer and $z$ a complex number of magnitude one with non-negative imaginary part. Then for all $\lambda \in \mathbb{R}$ with $\lambda >0,$
\begin{align}
\left| \frac{H_n'(\lambda z)}{H_n(\lambda z)}+\frac{1}{2\lambda z} \right| =O(\lambda^{-2})
\end{align}
as $\lambda \to \infty.$
\end{corollary}

The following proposition gives a formula for the Wronskian of $J_n$ and $Y_n$ and can be found, for example, in \cite{abram}.
\begin{proposition}
Let $n$ be an integer and $z$ be a complex number which is not a non-positive purely-imaginary number. Then
\begin{align}\label{eq:bes_wronsk}
J_n(z)Y_n'(z) -J_n'(z) Y_n(z) = \frac{2}{\pi z}.
\end{align}
\end{proposition}
A similar result holds for the Wronskian of $J_n$ and $H_n.$ Its proof is an immediate consequence of the definition of $H_n$ and the preceding proposition.
\begin{corollary}
Let $n$ be an integer and $z$ be any complex number which is not a non-positive purely-imaginary number. Then
\begin{align}\label{eq:beshank_wronsk}
J_n(z)H_n'(z) -J_n'(z) H_n(z) = \frac{2i}{\pi z}.
\end{align}
\end{corollary}

\subsection{Basic lemmas}

Lemma \ref{lem:perov} provides a variant of Gronwall's inequality  (see, for example, \cite{perov}).
\begin{lemma}\label{lem:perov}
Suppose that $A,B\in \mathbb{C}$ and $F,G$ are two positive real numbers. Suppose further that $f:[a,b]\to \mathbb{C}$ and $g:[a,b] \to \mathbb{C}$ are two functions such that $|f(r)| \le F<\infty$ and $|g(r)| \le G < \infty$ for all $a \le r \le b$ and that $w:[a,b] \to \mathbb{C}$ is the function defined by
\begin{align}
w(r) = A\int_r^b \left(w(x)+f(x) \right)w(x)\,{\rm d}x +B\int_r^b g(x)\,{\rm d}x.
\end{align}
If the constants $A,B,F,$ and $G$ are such that
\begin{align}
\left( 1 + \frac{F}{|B|G}\right)>\frac{1}{2}e^{\frac{|A|}{F}(b-a)},
\end{align}
then for all $a \le r \le b,$
\begin{align}
|w(r)| \le \frac{4}{F}. 
\end{align}
\end{lemma}

\iffalse
\begin{proof}

We begin by observing that
\begin{align}
|w(r)| = |A|\int_r^b \left(|w(x)|+2 |f(x)| \right)|w(x)|\,{\rm d}x +|B|\int_r^b |g(x)|\,{\rm d}x
\end{align}
and hence if we define the function $v:[a,b] \to \mathbb{R}$ by
\begin{align}\label{eqn:vint_eqn}
v(r) = |A|\int_r^b \left(v(x)+2 F \right)v(x)\,{\rm d}x +|B| G,
\end{align}
it is straightforward to show that $v(r) \ge |w(r)|$ for all $a \le r \le b.$ Moreover, differentiating (\ref{eqn:vint_eqn}) with respect to $r,$ we obtain the following initial value problem
\begin{align}\label{eqn:sepv}
v'(r) &= -|A|  \,(v(r)+2F)\,v(r)\\
v(b) &= |B| G.
\end{align}
Solving the separable ODE (\ref{eqn:sepv}) we obtain
\begin{align}
\log\left( 1+\frac{2F}{v(r)}\right) = -\frac{|A|}{2F}(b-r) +\log\left(1 + \frac{2F}{|B|G}\right).
\end{align}
Isolating the previous expression for $v$ yields
\begin{align}
v(r) = \frac{1}{2F} \frac{e^{\frac{|A|}{2F}(b-r)}}{\left( 1 + \frac{2F}{|B|G}\right)-e^{\frac{|A|}{2F}(b-r)}}
\end{align}
for all $a \le r \le b$ provided that
\begin{align}
\left( 1 + \frac{2F}{|B|G}\right)>e^{\frac{|A|}{2F}(b-a)}.
\end{align}
In particular, if 
\begin{align}
\left( 1 + \frac{2F}{|B|G}\right)>\frac{1}{2}e^{\frac{|A|}{2F}(b-a)}.
\end{align}
then
\begin{align}
v(r) \le \frac{2}{F}.
\end{align}
\end{proof}
\fi

The following lemma provides a bound on the solutions to a certain initial value problem arising in the WKB approximation of solutions to inhomogeneous Helmholtz equations in one dimension (for proofs see, for example, \cite{fedor,yuchen}).
\begin{lemma}\label{lem:wkb_lem}
Suppose that $T$ and $M$ are positive constants and let $K$ be the set defined by 
$$K = \{k \in \mathbb{C} \,:\, {\rm Im}\, k \ge -M, \,|k| \ge 1 \}.$$
 Suppose further that $\eta:[0,T] \times K \to \mathbb{R}$ is an absolutely continuous function uniformly bounded on $[0,T]\times K.$ Let $w:[0,T] \times K \to \mathbb{C}$ be the solution to the following initial value problem
\begin{align}\label{eqn:w_eqn_lem}
&w''(t,k)-2i{k} w'(t,k) = \eta(t,k) w(t,k)\\
&w(0,k) = 1\\
&w'(0,k) = 0,
\end{align}
where $'$ denotes differentiation with respect to $t.$ Then there exist
constants $C_1$ and $C_2$ depending on $\eta,$ $M$ and $T$ but which are
independent of $t$ and $k$ such that
\begin{align}
\left|w(t,k) -1+\frac{1}{2ik} \int_0^t \eta(\tau,k)\,{\rm d}\tau\right|  \le \frac{C_1}{|k|^2},\\
\left|w'(t,k) \right|  \le \frac{C_2}{|k|},
\end{align}
for all $t \in[0,T].$
\end{lemma}

\section{Mathematical apparatus}
In this section we establish properties of the impedance used in the construction of the trace formula. 
\begin{proposition}\label{prop_imped}
Let $n$ be a non-negative integer, and $k$ be a non-zero complex number with non-negative imaginary part. Then the function $u_n$ (see (\ref{eqn:1d})) has no zeros on the interval $0<r<\infty$ and hence neither does $\psi_n$ defined in (\ref{eqn:psi_def}).
\end{proposition}
\begin{proof}
First suppose that ${\rm Im} \,k^2 \neq 0.$ Note that $\lim_{r\to\infty} u_n(r) = 0$ and hence
\begin{align}
r u_n(r)\bar{u}_n'(r) - r u_n'(r) \bar{u}_n(r) &= \int_r^\infty \left(\bar{u}_n(x)\, (xu_n'(x))'- u_n(x)\, (x\bar{u}_n'(x))' \right){\rm d}x.
\end{align}
The substitution of (\ref{eqn:1dm}) into the right-hand side of the previous equation yields
\begin{align}
r u_n(r)\bar{u}_n'(r) - r u_n'(r) \bar{u}_n(r) &= {\rm Im}\, k^2 \int_r^\infty x |u_n(x)|^2 {\rm d}x,
\end{align}
and hence clearly $u_n$ cannot vanish for any $r>0.$ 

Next suppose $k = i \kappa,$ for $\kappa  \in \mathbb{R}^+.$ Then $u_n$ is
real and
\begin{align}\label{eqn:un101}
\int_r^\infty \left(u_n(x) (x u_n'(x))'-\left(\kappa^2(1+q(x))+ \frac{n^2}{x^2} \right) x u_n^2(x)\right)\,{\rm d}x=0.
\end{align}
Integrating (\ref{eqn:un101}) by parts gives
\begin{align}
r u_n(r) u_n'(r) = - \int_r^\infty \left(x (u_n'(x))^2 +\left(\kappa^2(1+q(x))+ \frac{n^2}{x^2} \right) x u_n^2(x)\right)\,{\rm d}x
\end{align}
which implies that $u_n(r) \neq 0$ for all $r >0.$

Finally, suppose that $k \in \mathbb{R}.$ By Remark \ref{lem_soln_char} there exists a constant $\beta$ such that for all $r>b,$ $u_n(r) = \beta H_n(kr).$ If $k >0$ then $J_n(kr)$ and $Y_n(kr)$ are real in which case ${\rm Re}\,(u_n(r)/\beta) = J_n(kr)$ and ${\rm Im}\, (u_n(r)/\beta) = Y_n(kr).$ Substituting these expressions into (\ref{eq:bes_wronsk}), we obtain
\begin{align}
{\rm Re}\left( \frac{u_n'(r)}{\beta}\right){\rm Im}\left( \frac{u_n(r)}{\beta}\right)-{\rm Re}\left( \frac{u_n(r)}{\beta}\right){\rm Im}\left( \frac{u_n'(r)}{\beta}\right) = \frac{2}{\pi k r}
\end{align}
and hence $u_n(r) \neq 0$ for all $r > 0.$ An almost identical argument applies to the case where $k<0.$
\end{proof}
\begin{theorem}\label{thm:ana_dep}
Let $q:(0,\infty) \rightarrow [q_0,q_1]$ with $-1<q_0\le q_1 <\infty$ be a continuous function supported on the interval $[a,b]$ with $0<a<b<\infty.$ For all non-negative integers $n$ and for all real numbers $r >0$ the impedance $\phi_n(r,k)$ is an analytic function of $k$ everywhere in the complex upper-half plane.
\end{theorem}
\begin{proof}
By Proposition \ref{prop_imped} the impedance is well-defined for all non-zero $k$ with non-negative imaginary part. Theorem \ref{thm:ana_dep} follows from the analytic dependence on parameters of solutions to ordinary differential equations (see, for example, \cite{codd}).
\end{proof}

The following theorem describes the behaviour of the impedance in the vicinity of $k =0.$
\begin{theorem}\label{thm:near_k0}
Let $0<a<b<\infty,$ and $n$ be any non-negative integer. Then
\begin{align}
\phi_n(r,k) =  \frac{(\sqrt{r}H_n(kr))'}{ik\sqrt{r} H_n(kr)} +O(k),
\end{align}
as $k \to 0$ in the complex upper half-plane (including the real axis).
\end{theorem}
\begin{proof}
Let $n$ be a non-negative integer and $k$ a non-zero complex number with non-negative imaginary part. Consider the function $w:[a,b]\times \{k \neq 0 \in \mathbb{C} : \,{\rm Im}\,k \ge 0\} \to \mathbb{C}$ defined by
\begin{align}\label{eqn:wdeff2}
w(r,k) = \phi_n(r,k)-\frac{(\sqrt{r}H_n(kr))'}{ik \sqrt{r} H_n(kr)}.
\end{align}
We begin by observing that by Corollary \ref{cor:ricat_eqn}, $w$ satisfies the integral equation
\begin{align}\label{eqn:wint}
w(r,k) =  ik \int_r^b \left(w(x,k) \left(w(x,k)+2\frac{(\sqrt{r}H_n(kx))'}{ik \sqrt{r} H_n(kx)}\right) +q(x)\right) \,{\rm d}x,
\end{align}
for all $a \le r \le b.$ Next we note that by Remark \ref{rem:hnearzero}, if $|k| < C_n/b$ then
\begin{align}\label{eqn:whbnd}
\left| \frac{(\sqrt{r}H_n(kr))'}{ik \sqrt{r} H_n(kr)}\right| &\le \frac{4(n+1)}{|k|a}+\frac{1}{2|k|a}\nonumber\\
& \le \frac{4(n+2)}{|k|a}
\end{align}
for all $a\le r \le b.$ Applying Lemma \ref{lem:perov} to the integral equation (\ref{eqn:wint}) with $|A| = |B| = |k|,$ $F = 8(n+2)/(|k|a)$ and $G = |q_0|+|q_1|,$ and using the bound (\ref{eqn:whbnd}) we obtain
\begin{align}\label{eqn:wbnd}
|w(r)| \le \frac{|k|a}{2(n+2)}
\end{align}
for all $a \le r \le b$ provided that
\begin{align}\label{eqn:prelim_bnd}
1+\frac{8(n+2)}{|k|^2a(|q_0|+|q_1|)} \ge \frac{1}{2} e^{\frac{a|k|^2}{8(n+2)}(b-a)}
\end{align}
and $|k| \le C_n/b.$ Substituting $|k| < C_n/b$ into the right-hand side of (\ref{eqn:prelim_bnd}) and rearranging yields
\begin{align}\label{eqn:wcond}
|k| \le \min\left\{ \frac{C_n}{b} ,\sqrt{\frac{16(n+2)}{a(|q_0|+|q_1|)}}\sqrt{\frac{1}{| e^{{C_n^2}/{(8(n+2))}}-2|}} \right\}.
\end{align}
Combining (\ref{eqn:wbnd}), (\ref{eqn:wcond}) and the definition of $w$ in (\ref{eqn:wdeff2}) we see that for all $r \in [a,b],$
\begin{align}
\phi_n(r,k) = \frac{(\sqrt{r}H_n(kr))'}{ik\sqrt{r} H_n(kr)} + O(k)
\end{align}
as $k \to 0$ in the complex upper half-plane (including the real axis).
\end{proof}

The following theorem describes the behaviour of the impedance at large frequencies.
\begin{theorem}\label{thm:WKB_form}
Suppose $q \in C^2(\mathbb{R})$ is a compactly supported function on the interval $[a,b].$ Moreover, suppose that there exist constants $q_0$ and $q_1$ such that $-1<q_0\le q(r) \le  q_1<\infty$ for all $r \in [a,b].$ Let $\phi_n$ be the impedance defined in (\ref{def_phin}). Then
\begin{align}\label{eqn:WKB_fin}
\phi_n(r,k) = \sqrt{1+q(r)}-\frac{1}{4ik}\frac{q'(r)}{1+q(r)}+O\left(\frac{1}{k^2}\right),
\end{align}
as $k\to \infty,$ $\Im (k) \ge 0.$
\end{theorem}
\begin{proof}
The proof is a slight modification of the standard analysis of the WKB
approximation applied to equation (\ref{eqn:psim_def}) (see
\cite{yuchen,fedor} for example). Indeed, let $s(r) = \sqrt{1+q(r)}$ and
define $t:[a,b] \to [0,\infty)$ by
\begin{align}
t(r) = \int_r^b s(\tau)\,{\rm d}\tau,
\end{align}
observing that 
\begin{align}
t(a) \le \sqrt{1+q_1} (b-a).
\end{align}
We set $T = t(a)$ and define $\tilde{k} \in \mathbb{C}$ by
\begin{equation}
\tilde{k} = k\left(\frac{H_n'(kb)}{H_n(kb)}+\frac{1}{2kb} \right).
\end{equation}
It follows from Corollary \ref{cor:hank_wkb} that $|\tilde{k} -k| = O(|k|^{-1})$ as $|k|\to \infty$ anywhere in the upper half-plane.

Next we define $w(t,k)$ implicitly by
\begin{align}\label{eqn:w_defn}
\psi_n(r(t),k) = {\sqrt{kb}}H_n(kb)\,\frac{e^{-i\tilde{k}t}w(t,k)}{\sqrt{s(r(t))}}.
\end{align}
For notational convenience in the following we will suppress the dependence of $w$ on $k$ and write $w(t)$ in place of $w(t,k).$

%\begin{align}
%&e^{-i\tilde{k}t}  w(t) =  \frac{1}{\mu}\psi_n(r(t)) \sqrt{n(r(t))}\\
%& -i\tilde{k} \mu e^{-i\tilde{k}t} w(t)+\mu e^{-i\tilde{k}t} w'(t) =  -\frac{1}{n} \left(\sqrt{n} \psi_n' +\frac{n'}{2n^\frac{1}{2}} \psi \right) \\
%&  -\tilde{k}^2 \mu e^{-i\tilde{k}t} w(t) -2i\tilde{k} \mu e^{-i\tilde{k}t} w'(t)+\mu e^{-i\tilde{k}t} w''(t) = \frac{1}{n} %\left( \frac{\psi_n''}{\sqrt{n}}+ \psi_n \frac{n''}{2n^\frac{3}{2}}-\frac{3 (n')^2}{4 n^\frac{5}{2}}\psi_n\right) \\
%&  -\tilde{k}^2 \mu e^{-i\tilde{k}t} w(t) -2i\tilde{k} \mu e^{-i\tilde{k}t} w'(t)+\mu e^{-i\tilde{k}t} w''(t) = \\
%&\hspace{1 cm} \left(-k^2 +\frac{g(x(t))}{n^2(x(t))}+  \frac{n''}{2n^2}-\frac{3 (n')^2}{4 n^3}\right)w e^{-i\tilde{k}t} 
%\end{align}

After inserting (\ref{eqn:w_defn}) into (\ref{eqn:psim_def}), clearly $w$ satisfies the following initial value problem
\begin{align}\label{eqn:w_eqn}
&w''(t)-2i\tilde{k} w'(t) = \eta(t) w(t),\\
&w(0) = 1,\nonumber\\
&w'(0) = 0,\nonumber
\end{align}
where
\begin{align}
\eta(t) =  \left(\tilde{k}^2-k^2 +\frac{n^2-\frac{1}{4}}{r^2(t)\,s^2(r(t))}+  \frac{s''(r(t))}{2s^2(r(t))}-\frac{3 (s'(r(t)))^2}{4 s^3(r(t))}\right).
\end{align}
We note that $\tilde{k}^2-k^2 = O(1)$ as $k\to \infty$ in the upper half-plane and thus that $\eta$ is an absolutely continuous function on $[0,T]$ and is bounded uniformly in $k$ and $t$ for all $\Im{k} \ge 0, |k|>1$ and $t\in [0,T].$ Moreover, since $k- \tilde{k} = O(|k|^{-1}),$ there exists some constant $M$ such that ${\rm Im}\, \tilde{k} \ge -M$ for all $k \in \mathbb{C}$ such that $\Im{k} \ge 0, |k|>1.$ 

Applying Lemma \ref{lem:wkb_lem} to the initial value problem (\ref{eqn:w_eqn}) we obtain
\begin{align}
\left|w(t) -1 +\frac{1}{2i\tilde{k}}\int_0^t  \eta(\tau) \,{\rm d}\tau \right| &= O(|\tilde{k}|^{-2}),\\
\left| w'(t) \right| &= O(|\tilde{k}|^{-1}),
\end{align}
for all $t \in [0,T].$

Finally, it follows from the definition of $w$, see equation (\ref{eqn:w_defn}), that
\begin{align}
\frac{\psi_n'(r)}{\psi_n(r)} = i\tilde{k} \sqrt{1+q(r)}+\frac{s'(r)}{2s(r)} + O(|\tilde{k}|^{-1})
\end{align}
as $\tilde{k} \to \infty$ anywhere in the upper half-plane.
\end{proof}

\section{The trace formula}
In this section we present a trace formula for the impedance which is the principal analytic tool used in the inversion algorithm.
\begin{theorem}\label{thm:trac_form}
Suppose that $q \in C^2(\mathbb{R})$ is a compactly supported function on the interval $[a,b]$ and that there exist constants $q_0$ and $q_1$ such that $-1<q_0\le q(r) \le  q_1<\infty$ for all $r \in [a,b].$ Let $\phi_n$ be the impedance defined in Definition \ref{def_imped}. Then
\begin{align}\label{eqn:trace_formula}
\frac{q'(r)}{1+q(r)} = \frac{4}{\pi}\int_{-\infty}^\infty \left(\phi_n(r,k)-\frac{(\sqrt{r}H_n(kr))'}{ik H_n(kr)}+1-\sqrt{1+q(r)} \right) \,{\rm d}k.
\end{align}
\end{theorem}
\begin{proof}
For $a \le r \le b$ define the function $f_r : \{k \neq 0 \in \mathbb{C}\, : \, {\rm Im} \,k \ge 0\} \to \mathbb{C}$ by
\begin{align}
f_r(k) = \phi_n(r,k)-\frac{(\sqrt{r}H_n(kr))'}{ik H_n(kr)}+1-\sqrt{1+q(r)}.
\end{align}
By Theorem \ref{thm:ana_dep}, for all $a\le r \le b$ the function $f_r$ is analytic in the upper half-plane and hence if $\Omega$ is any positive real number then
\begin{align}
\int_{-\Omega}^\Omega f_r(k)\,{\rm d}k = -i\Omega\int_{0}^\pi f_r(e^{i\theta}\Omega) e^{i\theta}\,{\rm d}\theta.
\end{align}
Substituting the asymptotic expansion of $\phi_n$  from equation (\ref{eqn:WKB_fin}) into the previous expression yields
\begin{align}
\int_{-\Omega}^\Omega f_r(k)\,{\rm d}k ={i\Omega}\int_0^\pi \frac{1}{4i \Omega e^{i\theta}} \frac{q'(r)}{1+q(r)}e^{i\theta}\,{\rm d}\theta + O\left( \frac{1}{\Omega}\right).
\end{align}
Taking the limit as $\Omega \to \infty$ completes the proof.
\end{proof}

The following corollary is an immediate consequence of Theorem \ref{thm:trac_form} and Lemma \ref{lem:ricat_eqn}, and is the basis for the reconstruction algorithm described in Section \ref{sec:rec_algo}.
\begin{corollary}
Suppose that $q \in C^2(\mathbb{R})$ is a compactly supported function on
the interval $[a,b]$ and that there exist constants $q_0$ and $q_1$ such
that $-1<q_0\le q(r) \le  q_1<\infty$ for all $r \in [a,b].$ Then $\phi_n$
and $q$ satisfy the following system of integro-differential equations
\begin{align}\label{eqn:trace_syst}
\phi_n'(r,k) &= -ik\phi_n^2(r,k)-ik (1+q(x)) +i\frac{n^2-\frac{1}{4}}{kr^2}, \quad k \in \mathbb{R}\\
\frac{q'(r)}{1+q(r)} &= \frac{4}{\pi}\int_{-\infty}^\infty \left(\phi_n(r,k)-\frac{(\sqrt{r}H_n(kr))'}{ik H_n(kr)}+1-\sqrt{1+q(r)} \right) \,{\rm d}k
\end{align}
for all $r \in [a,b]$ together with the initial conditions
\begin{align}
\phi_n(a,k) &= k\left(\frac{H_n'(ka)}{iH_n(ka)}+\frac{1}{2ika} \right),\quad\quad k \in \mathbb{R},\\
q(a) &= 0.
\end{align}
\end{corollary}

\section{Numerical algorithm and results}

\subsection{The reconstruction algorithm}\label{sec:rec_algo}
In this section we describe a reconstruction algorithm based on the trace formula derived in Theorem \ref{thm:trac_form}. As input it takes a non-negative integer $n$, an interval $[a,b]$ with $0<a<b<\infty,$ a spatial step size $h,$ a bandlimit $\Omega,$ and the number of frequency samples $N$ to use. As output the algorithm produces an approximation to the potential $q$ on the interval $[a,b].$

\begin{enumerate}
\item[Step 1.] Initialization: For $j=1,\dots,N$ let $f_j = 2 \Omega (j-1)/(N-1)-\Omega$ and $w_1=\frac{\Omega}{N},$ $w_N = \frac{\Omega}{N}$ and $w_j = \frac{2 \Omega}{N},$ $j=2,\dots,N-1.$ We note that this corresponds to an $N$-point trapezoidal quadrature rule on the interval $[-\Omega,\Omega].$  Set $q_0 = q(a) = 0.$ Set $\phi_{0,j} =\phi_n(a,f_j),$ for $j=1,\dots,N.$
\end{enumerate}

{\noindent For $\ell=0,\dots, (b-a)/h-1$}
\begin{enumerate}
\item[Step 2.] Set $r_\ell = a+\ell h$ and obtain $q'(r_\ell)$ via the formula
\begin{align}\label{eqn:dereval}
q'(r_\ell) =\frac{\pi (1+q_\ell)}{4}
\sum_{j=1}^N\left(\phi_{\ell,j}-\frac{\sqrt{r_\ell}H_n'(f_j r_\ell)}{i H_n(f_j
r_\ell)}-\frac{1}{2i f_j \sqrt{r_\ell}}+1-\sqrt{1+q_\ell} \right) w_j,
\end{align}
and compute ${q}_{\ell+1}$ via the formula
\begin{align}
{q}_{\ell+1} = q_\ell + h q'(r_\ell).
\end{align}
\item[Step 3.] For $j=1,\dots,N$ set
$$\phi_{\ell+1,j} = \phi_{\ell,j} + h \left(-if_j \phi_{\ell,j}^2-if_j(1+q_\ell)+ i \frac{n^2-\frac{1}{4}}{f_jr_\ell^2}\right).$$
\end{enumerate}
\begin{remark}
The above algorithm is first-order in $1/\Omega$ and $h$ and $1/N.$ In the next section we discuss modifications which improve its rate of convergence with respect to these parameters.
\end{remark}

\subsection{Numerical acceleration of convergence }
The algorithm presented in the previous section is first-order in the bandlimit $\Omega,$ the number of frequency samples $N,$ and the spatial step size $h,$ and is suitable for situations in which a few digits of relative precision are required for the reconstructions. If higher-precision reconstructions are required then the number of samples, the bandlimit and the number of spatial discretization points can become prohibitively large. In this section we outline straightforward modifications to the above algorithm which increase the rate of its convergence with respect to $\Omega,$ $h,$ and $N.$

\subsubsection{Dependence on $N$}
As written the algorithm uses the trapezoidal rule to approximate the
integral appearing in the trace formula (\ref{eqn:trace_formula}) over a
truncated interval $[-\Omega,\Omega].$ Theorem  \ref{thm:near_k0} guarantees
that computing the integral over this interval using trapezoid rule will
result in an error that decays like $1/N$ where $N$ is the number of
frequencies used. For $k\in \mathbb{R}$ away from zero the integrand is
smooth and hence any smooth quadrature rule such as Gauss--Legendre
quadratures or nested Gauss--Legendre quadratures can be used to obtain
arbitrarily high accuracy. Near $k=0$ the presence of terms depending on
$\log(k)$ cause singularities in the higher derivatives of the integrand
which necessitate the use of a different quadrature rule. In particular,
using generalized Gaussian quadratures \cite{gaussquad} we produced a
$35$-point quadrature rule which integrates all functions of the form
$$f_{m,n}(k) = k^m \log^n(k)$$
on the interval $0\le k \le 1/2$ for $m=1,2,\dots,18$ and $n= -10,-9,\dots,4$ to a relative precision of $10^{-16}.$ The resulting quadrature rule can be used to perform the integrals in the neighborhood of $k=0.$ Alternatively, one could use an endpoint corrected trapezoid rule \cite{endpoint} to evaluate the contribution of the integral in the vicinity of the origin.

Using this quadrature method, for any $\Omega>0$ and $0<r<\infty,$ integrals of the form 
\begin{align}
\frac{4}{\pi}\int_{-\Omega}^\Omega \left(\phi_n(r,k)-\frac{(\sqrt{r}H_n(kr))'}{ik H_n(kr)}+1-\sqrt{1+q(r)} \right) \,{\rm d}k
\end{align}
can be computed numerically to full machine precision with relatively few quadrature nodes (typically no more than 500 and often significantly fewer).

\subsubsection{Dependence on $\Omega$}
The method outlined in the previous section allows one to compute integrals of the form 
\begin{align}
\frac{4}{\pi}\int_{-\Omega}^\Omega \left(\phi_n(r,k)-\frac{(\sqrt{r}H_n(kr))'}{ik H_n(kr)}+1-\sqrt{1+q(r)} \right) \,{\rm d}k
\end{align}
accurately and with relatively few quadrature nodes. It does not, however,
eliminate the truncation error introduced by replacing the integral over the
entire real line in the system (\ref{eqn:trace_syst}) by the integral over
the finite interval $[-\Omega,\Omega].$ From Theorem \ref{thm:WKB_form} it
can be observed that the resulting error due to this truncation will decay
like $1/\Omega.$ In this section we describe a modification to the inversion
algorithm described above which produces faster convergence in $\Omega.$ The
principal tool is Richardson extrapolation.

For notational convenience we denote the real part of the integrand appearing in the trace formula (\ref{eqn:trace_formula}) by $F(r,k),$ noting that for any $0<r<\infty,$ $F(r,k) = F(r,-k).$ Specifically, $F:(0,\infty)\times \mathbb{R} \to \mathbb{R}$ is defined via the formula
\begin{align}
F(r,k) =\left(\phi_n(r,k)-\frac{(\sqrt{r}H_n(kr))'}{ik H_n(kr)}+1-\sqrt{1+q(r)} \right). 
\end{align}
We observe that the imaginary part can be neglected since by Lemma \ref{lem:phi_sym} the integral of the imaginary part vanishes provided the endpoints of integration are symmetric about $k=0.$ Additionally, Theorem \ref{thm:trac_form} guarantees that $F(r,k) = O(k^{-2})$ for large $k.$ In fact, for any fixed $r$ it has an asymptotic expansion in $k$ valid in the limit as $k$ goes to infinity; namely, there exist coefficients $A_2(r),A_4(r),\dots$ depending on the potential $q$ and the point $r,$ such that
\begin{align}
F(r,k) = \frac{A_2(r)}{k^2}+\frac{A_4(r)}{k^4} + \dots + \frac{A_{2m}(r)}{k^{2m}} + O(k^{-2m-2})
\end{align}
for any $m\ge 1.$

Thus
\begin{align}
2 \int_{-2\Omega}^{2\Omega} F(r,k)\,{\rm d}k-\int_{-\Omega}^\Omega F(r,k)\,{\rm d}k = O(\Omega^{-3}).
\end{align}
Rather than compute both integrals, this extrapolation can be performed by adjusting the frequency quadrature weights $w_j,$ $j=1,\dots,N.$ In addition, this extrapolation can be performed multiple times, each time increasing the rate of convergence by a factor of $\Omega^{-2}.$ Finally, we remark that it is not necessary to double the bounds of integration for each step of Richardson extrapolation: smaller ratios can be used at the expense of increasing the coefficients multiplying the integrals.

\subsection{Dependence on $h$}
The recovery algorithm described in Section \ref{sec:rec_algo} uses the forward Euler method to evolve both the impedance $\phi_n$ and the potential $q$ from the inner radius of the annulus $a$ to the outer radius $b,$ which produces an error decaying linearly in the step size $h.$ If one instead uses Heun's method for the evolution of the potential $q$ followed by the Crank--Nicholson method to evolve the equations for the impedance $\phi_{n},$ the  result is a second-order method in $h.$ 

\iffalse
Specifically, Step 2. in the reconstruction algorithm is replaced by the following three steps.
\begin{enumerate}
\item[Step 2.] Obtain $q'(r_\ell)$ via the formula
\begin{align}\label{eqn:dereval}
q'(r_\ell) =\frac{\pi (1+q_\ell)}{4} \sum_{j=1}^N\left(\phi_{\ell,j}-\frac{H_n'(f_j r_i)}{i H_n(f_j r_\ell)}-\frac{1}{2i f_j r_\ell}+1-\sqrt{1+q_\ell} \right) w_j,
\end{align}
and compute $\tilde{q}_{\ell+1}$ via the formula
\begin{align}
\tilde{q}_{\ell+1} = q_\ell + h q'(r_\ell).
\end{align}
\item [Step 3.] For each $j=1,\dots,N$ determine $\phi_{\ell+1,j}$ using one step of Crank-Nicholson. Specifically, solve the following quadratic equation for $\phi_{\ell+1,j}$
\begin{align}
2\frac{\phi_{\ell+1,j}-\phi_{\ell,j}}{h} &= -if_j\phi_{\ell,j}^2-if_j (1+q_\ell) +i\frac{n^2-\frac{1}{4}}{f_j r_\ell^2}\\
&\hspace{1 cm} -if_j\phi_{\ell+1,j}^2-if_j (1+q_{\ell+1}) +i\frac{n^2-\frac{1}{4}}{f_j r_{\ell+1}^2}\nonumber.
\end{align}

\item [Step 4.] Obtain $q'(r_{\ell+1})$ via formula (\ref{eqn:dereval}) with $\ell$ replaced by $\ell+1$ and compute $q_{\ell+1}$ via the formula
\begin{align}
{q}_{\ell+1} = q_\ell + \frac{h}{2} \left[ q'(r_\ell)+q'(r_{\ell+1}) \right].
\end{align}
\end{enumerate}
\fi
\begin{remark}
The above algorithm is second-order accurate in the step size $h$ both for the evolution of the impedance as well as for the evolution of the potential $q$; using Richardson extrapolation it is easy to obtain higher-order convergence in $h.$
\end{remark}

\subsection{Numerical results}
The algorithm described above, together with the modifications, was
implemented in Fortran and the results are summarized below. All code was
compiled in GFortran and run on a 2.7 GHz Apple laptop with 8 Gb of memory.
To avoid so-called {\it inverse crimes} the forward data was obtained by
solving the equation for the field $u_n$ given in equation (\ref{eqn:1d})
using a fourth-order Runge--Kutta method. We show both the effect of
increasing the order $n$ (Figure \ref{fig:highm_recovery}) as well as
changing the distance of the annulus from the origin (Figure
\ref{fig:far_recovery}). Finally, in Figure \ref{fig:square_pot} we show
recovery for a discontinuous potential.

\section{Conclusions and discussion}
In this paper, we present a procedure for solving inverse scattering
problems with radially-symmetric scattering potentials in the plane, given multifrequency impedance data.
The procedure is based on a new trace formula that relates the impedance of the field to
the scattering potential via an integro-differential equation, which can
then be solved to recover the scattering potential. 
Numerical results are included illustrating the accuracy and efficiency of
the method.

The approach of this paper extends directly to three-dimensional
radially-symmetric problems as well as waveguides with constant
cross-sectional parameters. 
Detailed analyses and numerical implementations in these cases will be
published at a later date.
The extension of this work to cases where the scattering potential is not
radially symmetric is currently being pursued. 

\vspace{1 cm}

J. H. and V. R. were both supported in part by ONR (grant no.\ N00014-14-1-0797) and AFOSR
(grant no.\ FA9550-16-1-0175).
V. R. was also supported in part by NSF (grant no.\ DMS-1952751).

\begin{figure}[p]
    \centering
    \begin{subfigure}[b]{0.45\textwidth}
        \includegraphics[width=\textwidth]{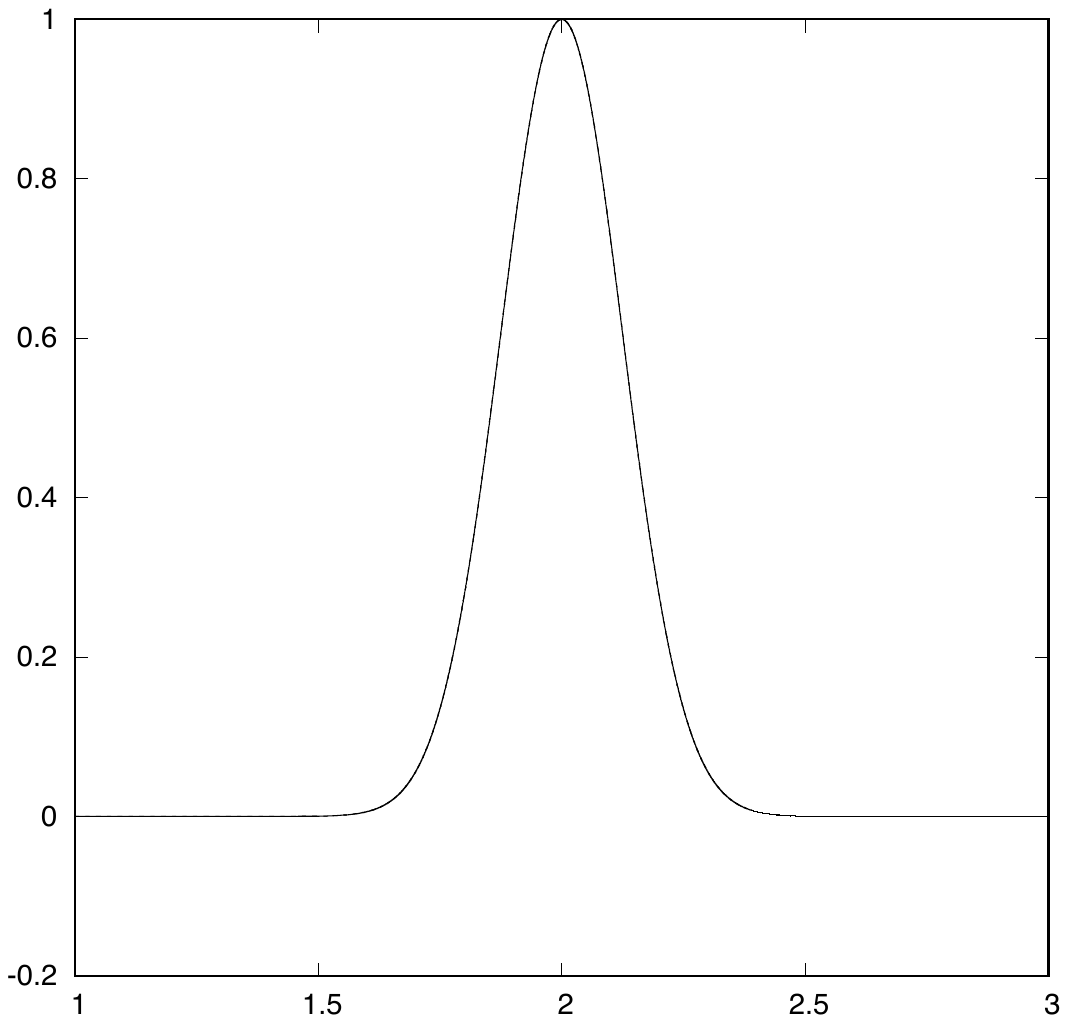}
        \caption{The exact and recovered\newline potential}
        \label{fig:bel_rec}
    \end{subfigure}
    \begin{subfigure}[b]{0.45\textwidth}
        \includegraphics[width=\textwidth]{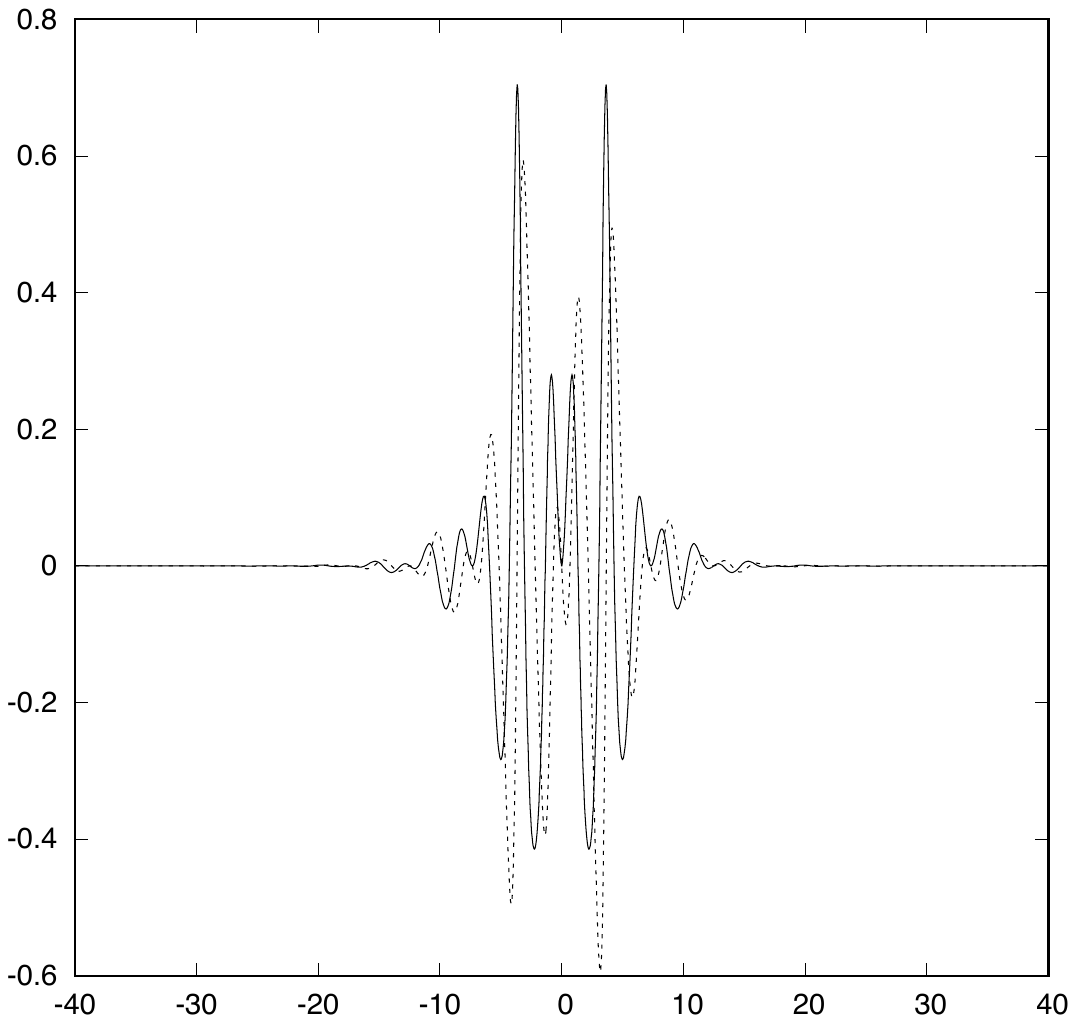}
        \caption{The initial data \newline $\phi_0 - (\sqrt{r} H_0(kr))'/(H_0(kr))$}
        \label{fig:bel_dat}
    \end{subfigure}
        \begin{subfigure}[b]{0.45\textwidth}
        \includegraphics[width=\textwidth]{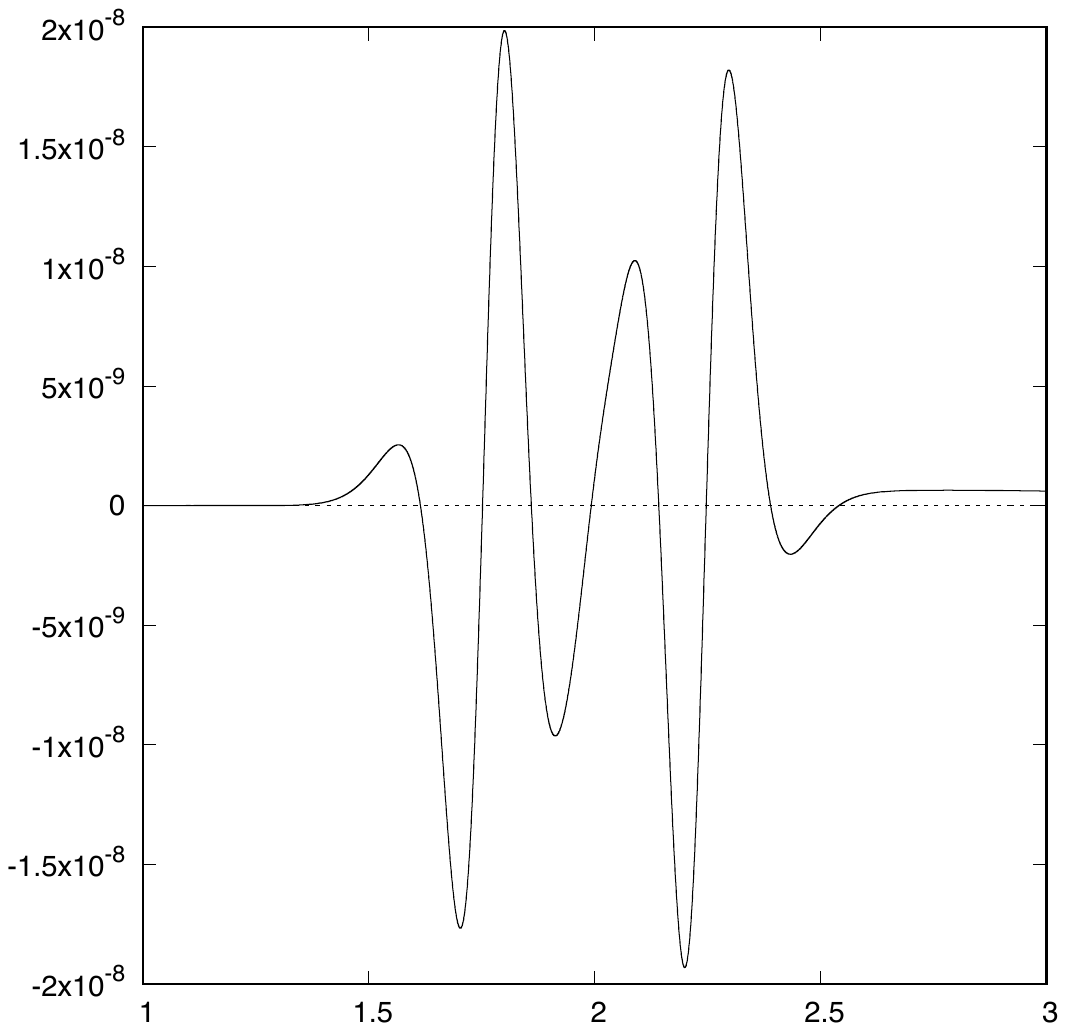}
        \caption{The recovery error}
        \label{fig:bel_err}
    \end{subfigure}
    \caption{Numerical results for a Gaussian bump $q(r) = e^{-32(x-2)^2}$ with $n=0.$ The time to generate the data was 54 seconds, and the time to solve was 140 seconds. The solve was done using $270$ frequencies in the range $[-160,160],$ and a spatial step size of $1/20 000.$}
\end{figure}

\begin{figure}[p]
    \centering
    \begin{subfigure}[b]{0.45\textwidth}
        \includegraphics[width=\textwidth]{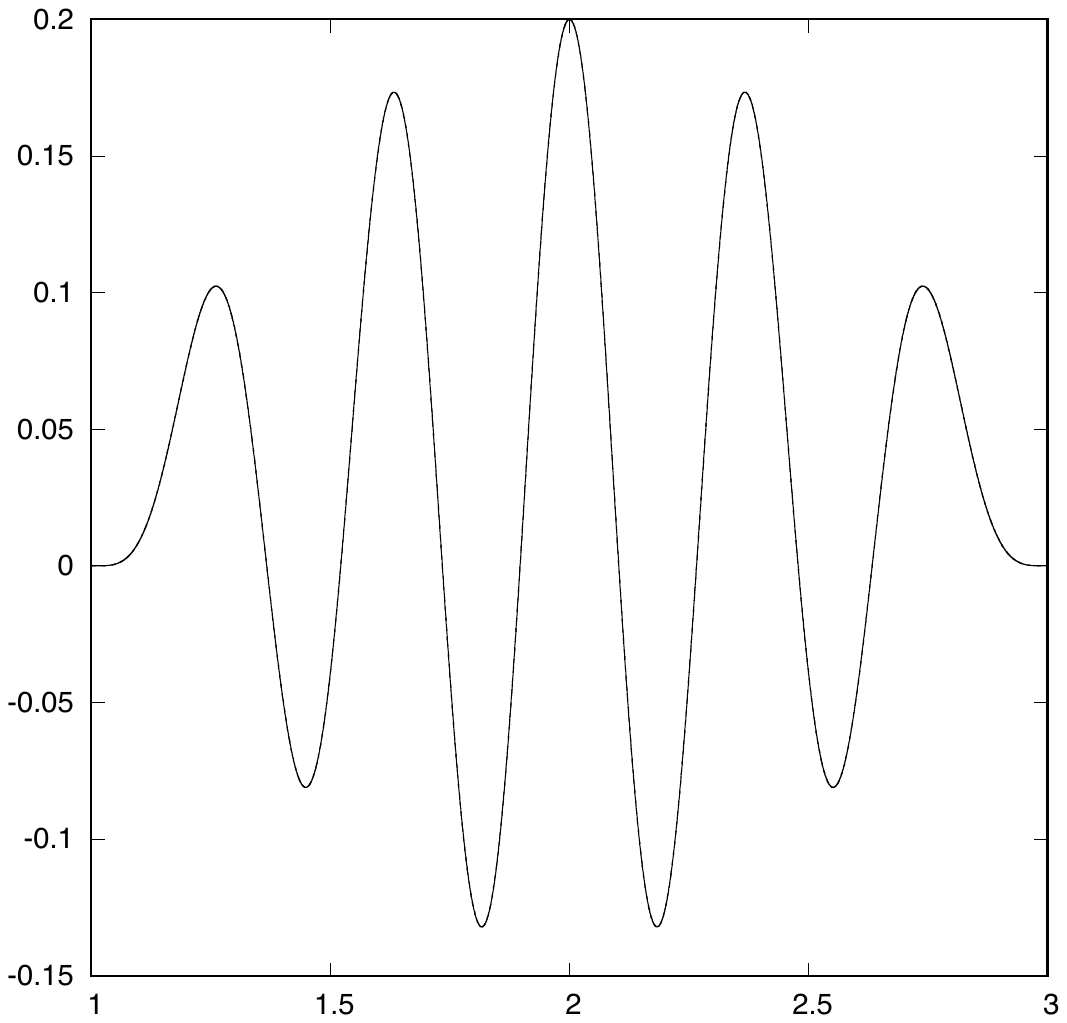}
        \caption{The exact and recovered\newline potential}
        \label{fig:sin_rec}
    \end{subfigure}
    \begin{subfigure}[b]{0.45\textwidth}
        \includegraphics[width=\textwidth]{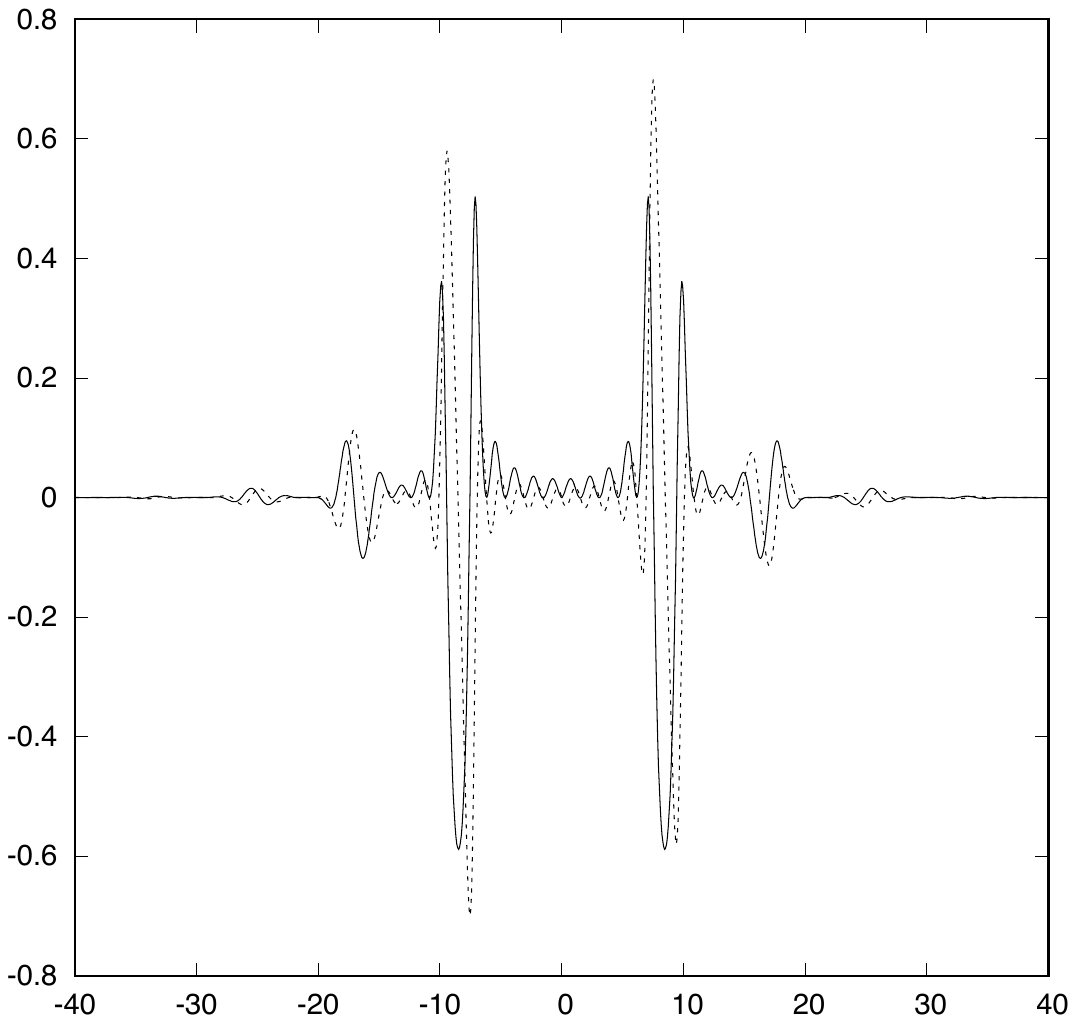}
        \caption{The initial data \newline $\phi_0 - (\sqrt{r} H_0(kr))'/(H_0(kr))$}
        \label{fig:sin_dat}
    \end{subfigure}
        \begin{subfigure}[b]{0.45\textwidth}
        \includegraphics[width=\textwidth]{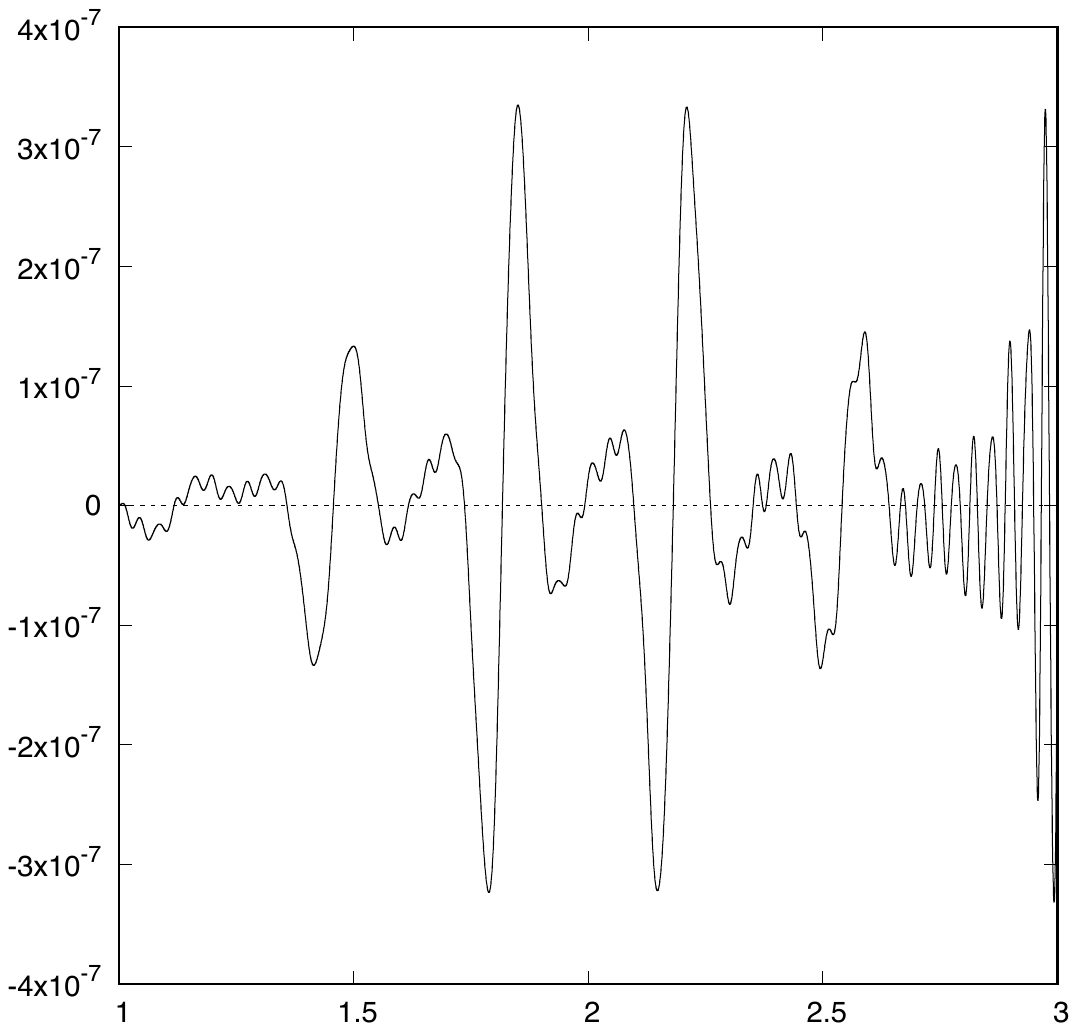}
        \caption{The recovery error}
        \label{fig:sin_err}
    \end{subfigure}
    \caption{Numerical results for the potential\\
   $q(r) = \frac{1}{10} \left[ \cos(a(r-2)\pi)+1\right]-\frac{a^2}{10 b^2} \left[1- \cos(b(r-2)\pi) \right] $ with $a=5,b=6$ and $n=0.$ The time to generate the data was 67 seconds, and the time to solve was 141 seconds. The solve was done using $270$ frequencies in the range $[-160,160]$ and a spatial step size of $1/20 000.$}
\end{figure}

\begin{figure}[p]
    \centering
    \begin{subfigure}[b]{0.45\textwidth}
        \includegraphics[width=\textwidth]{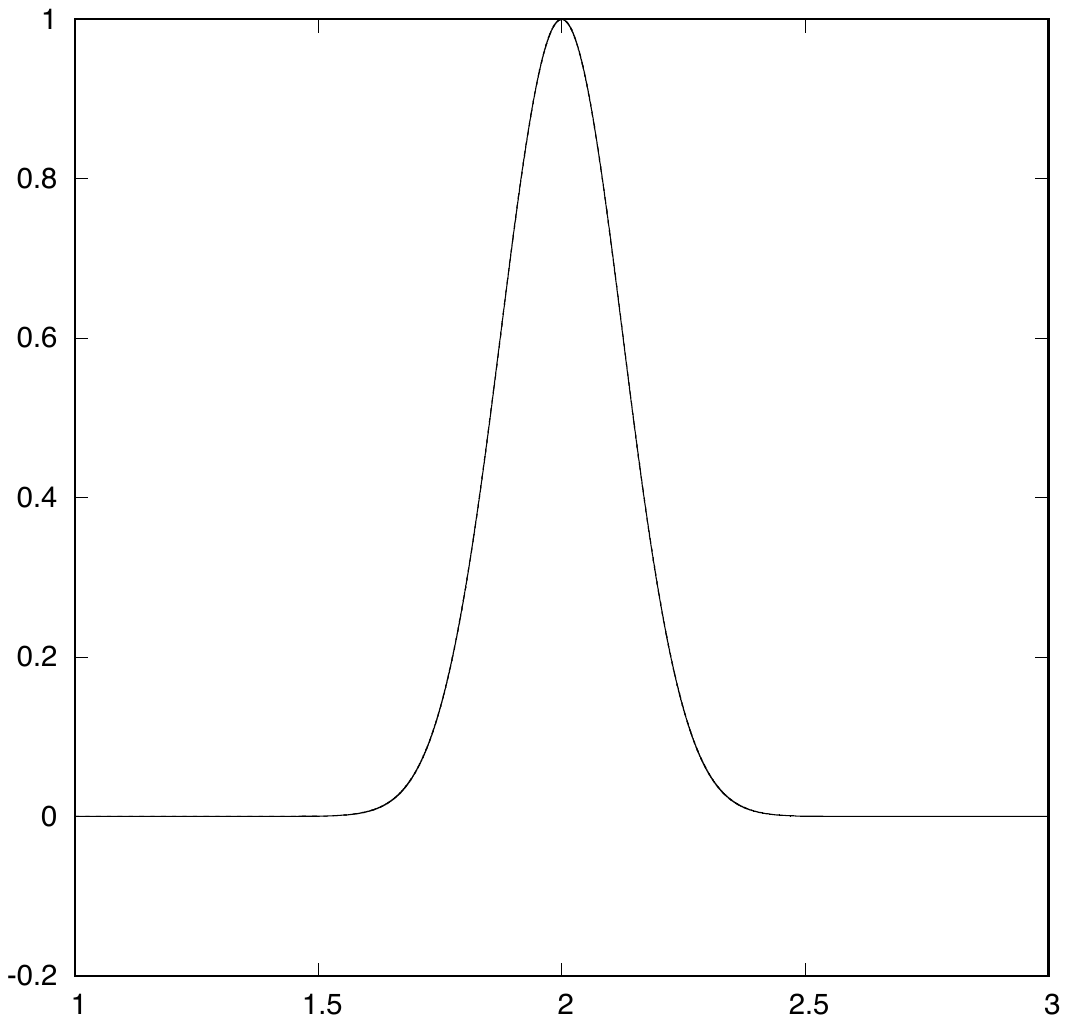}
        \caption{The exact and recovered\newline potential}
        \label{fig:belm_rec}
    \end{subfigure}
    \begin{subfigure}[b]{0.45\textwidth}
        \includegraphics[width=\textwidth]{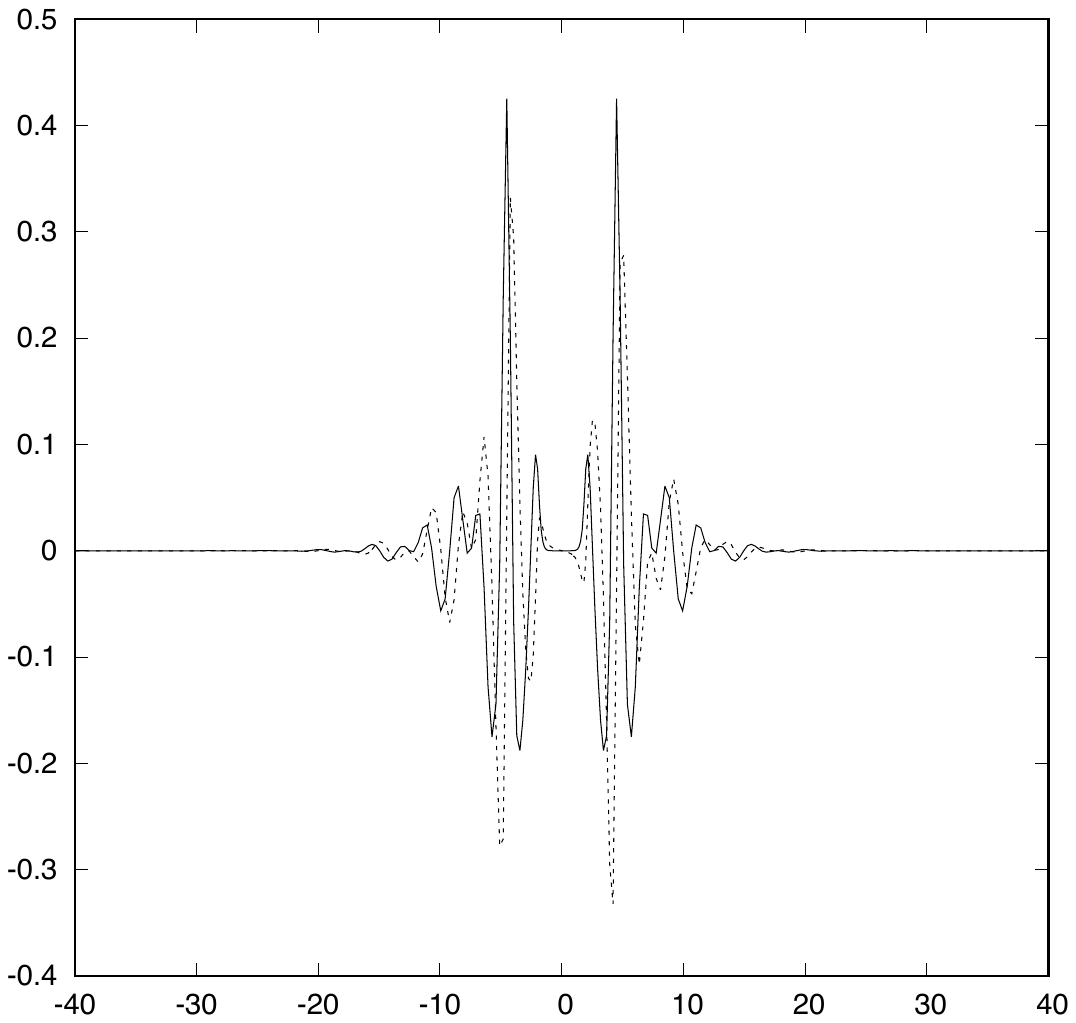}
        \caption{The initial data \newline $\phi_4 - (\sqrt{r} H_4(kr))'/(H_0(kr))$}
        \label{fig:belm_dat}
    \end{subfigure}
        \begin{subfigure}[b]{0.45\textwidth}
        \includegraphics[width=\textwidth]{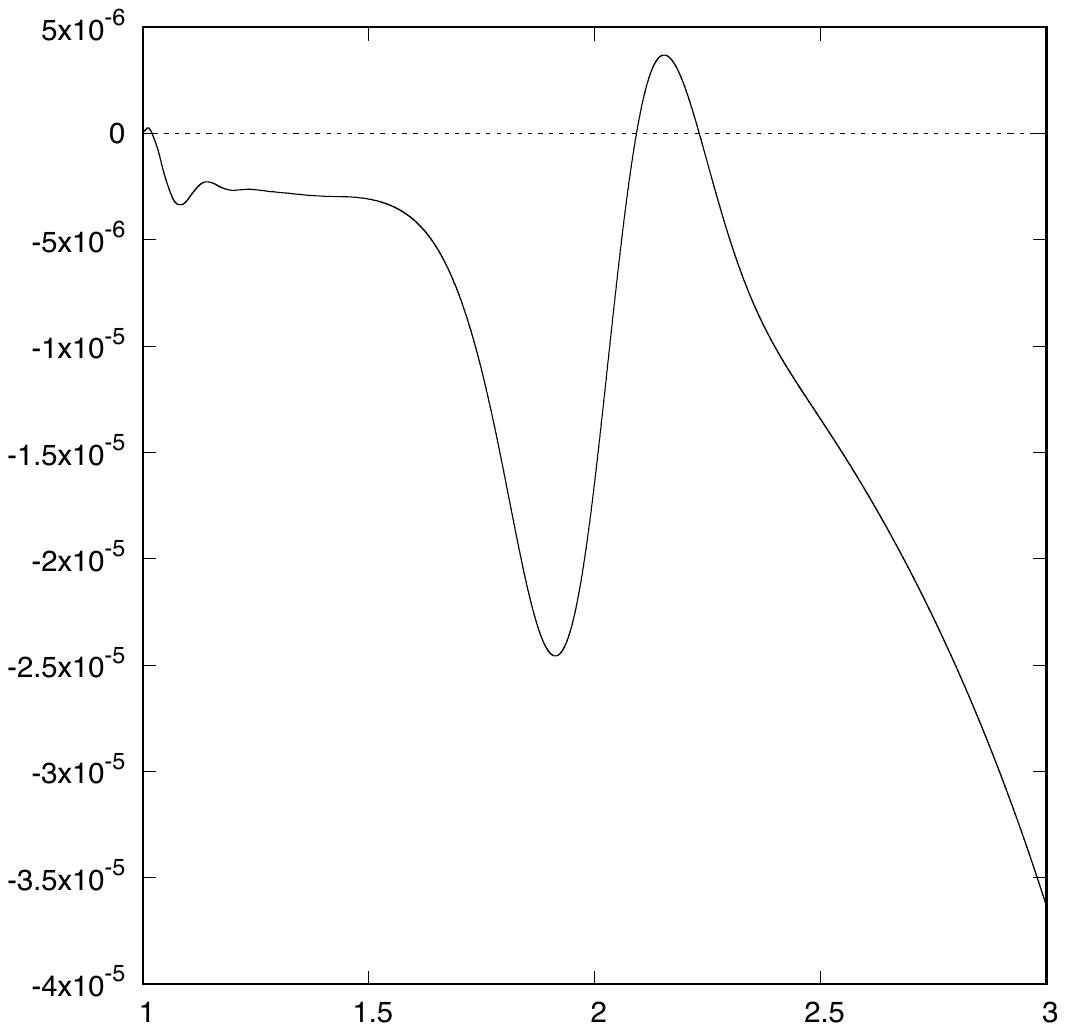}
        \caption{The recovery error}
        \label{fig:belm_err}
    \end{subfigure}
    \caption{Numerical results for a Gaussian bump $q(r) = e^{-32(x-2)^2}$ with $n=4.$ The time to generate the data was 37 seconds, and the time to solve was 88 seconds. The solve was done using $470$ frequencies in the range $[-240,240],$ and a spatial step size of $1/40 000.$}\label{fig:highm_recovery}
\end{figure}

\begin{figure}[p]
    \centering
    \begin{subfigure}[b]{0.45\textwidth}
        \includegraphics[width=\textwidth]{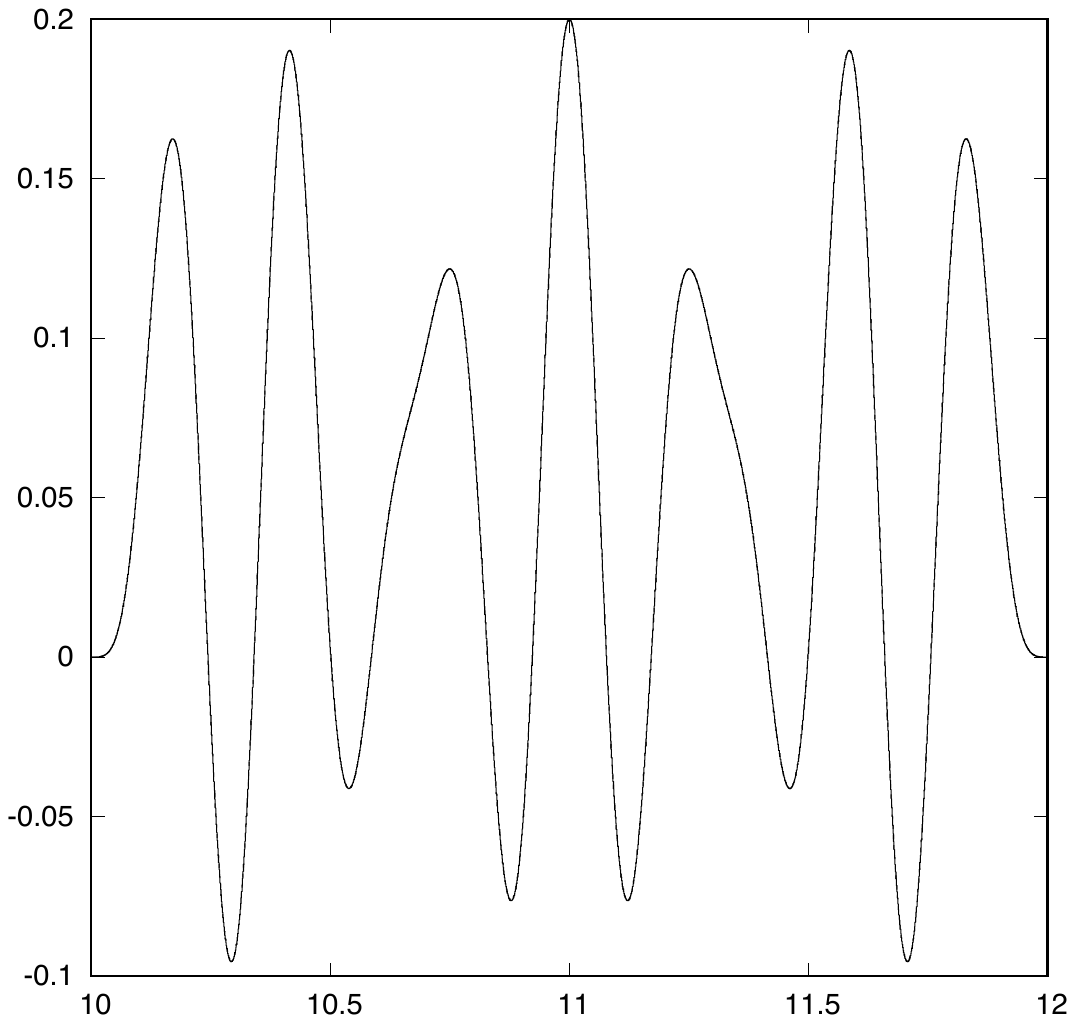}
        \caption{The exact and recovered\newline potential}
        \label{fig:sinf_rec}
    \end{subfigure}
    \begin{subfigure}[b]{0.45\textwidth}
        \includegraphics[width=\textwidth]{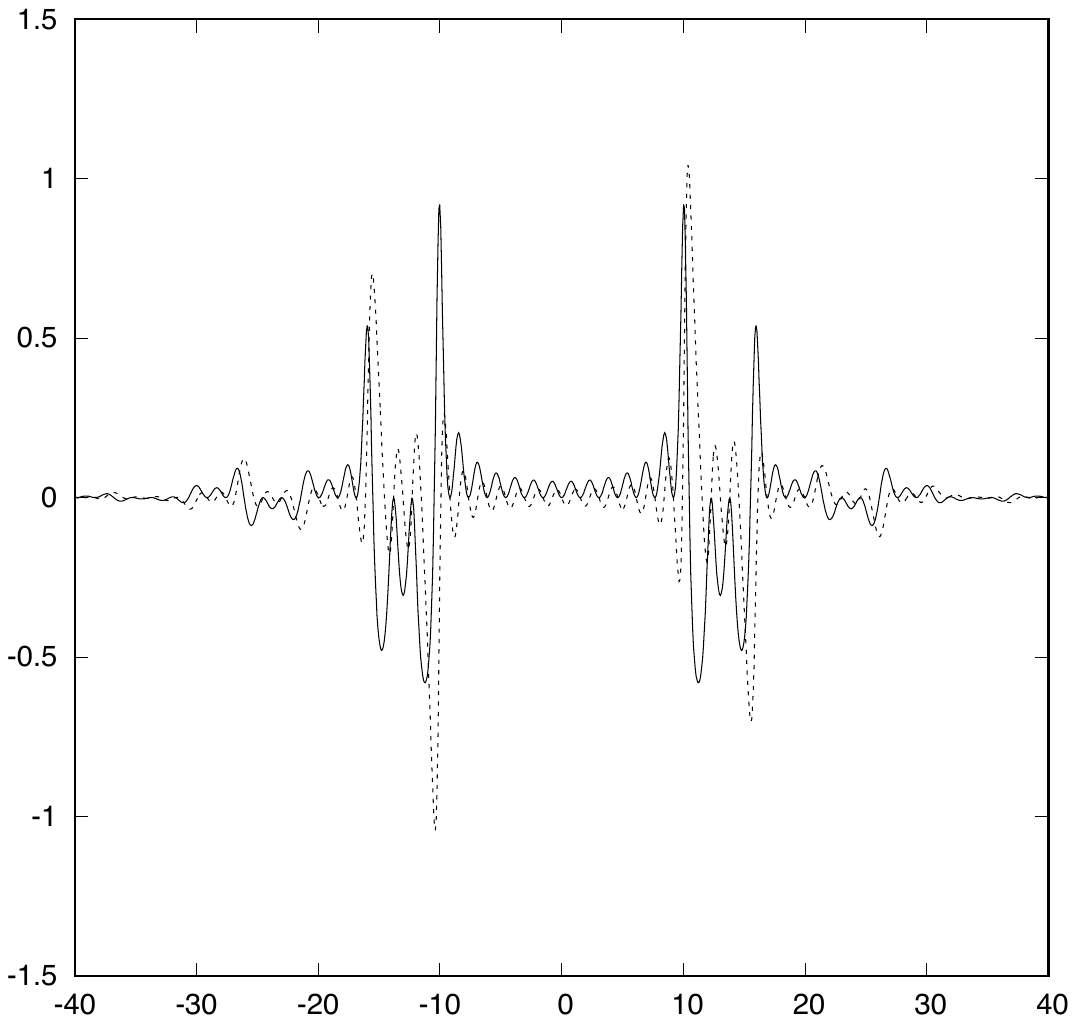}
        \caption{The initial data \newline $\phi_0 - (\sqrt{r} H_0(kr))'/(H_0(kr))$}
        \label{fig:sinf_dat}
    \end{subfigure}
        \begin{subfigure}[b]{0.45\textwidth}
        \includegraphics[width=\textwidth]{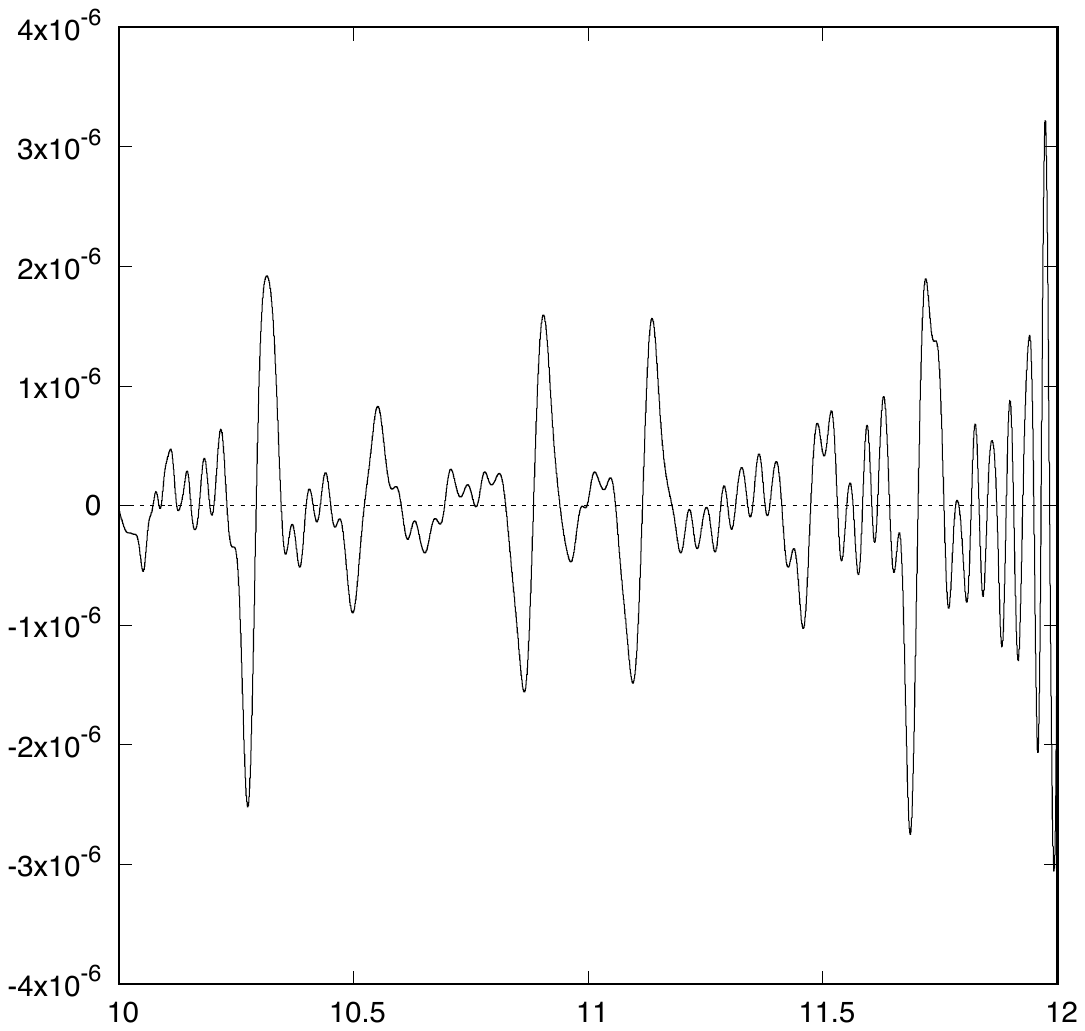}
        \caption{The recovery error}
        \label{fig:sinf_err}
    \end{subfigure}
    \caption{Numerical results for the potential\\
   $q(r) = \frac{1}{10} \left[ \cos(a(r-2)\pi)+1\right]-\frac{a^2}{10 b^2} \left[1- \cos(b(r-2)\pi) \right] $ with $a=9,b=10$ and $n=0.$ The time to generate the data was 62 seconds, and the time to solve was 130 seconds. The solve was done using $270$ frequencies in the range $[-160,160]$ and a spatial step size of $1/20 000.$}\label{fig:far_recovery}
\end{figure}

\begin{figure}[p]
    \centering
    \begin{subfigure}[b]{0.45\textwidth}
        \includegraphics[width=\textwidth]{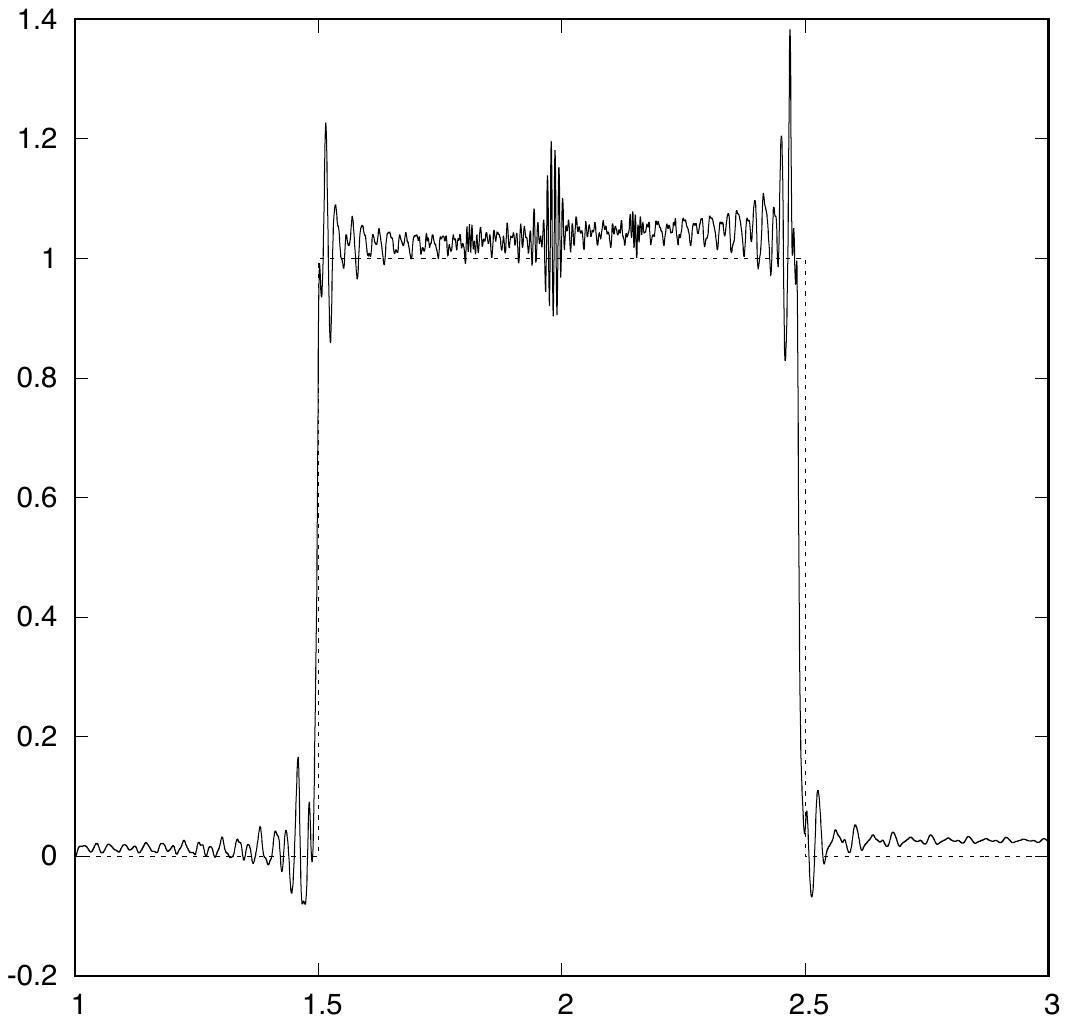}
        \caption{The exact (solid line) and\\ recovered (dashed line) potential}
        \label{fig:sqr_rec}
    \end{subfigure}
    \begin{subfigure}[b]{0.45\textwidth}
        \includegraphics[width=\textwidth]{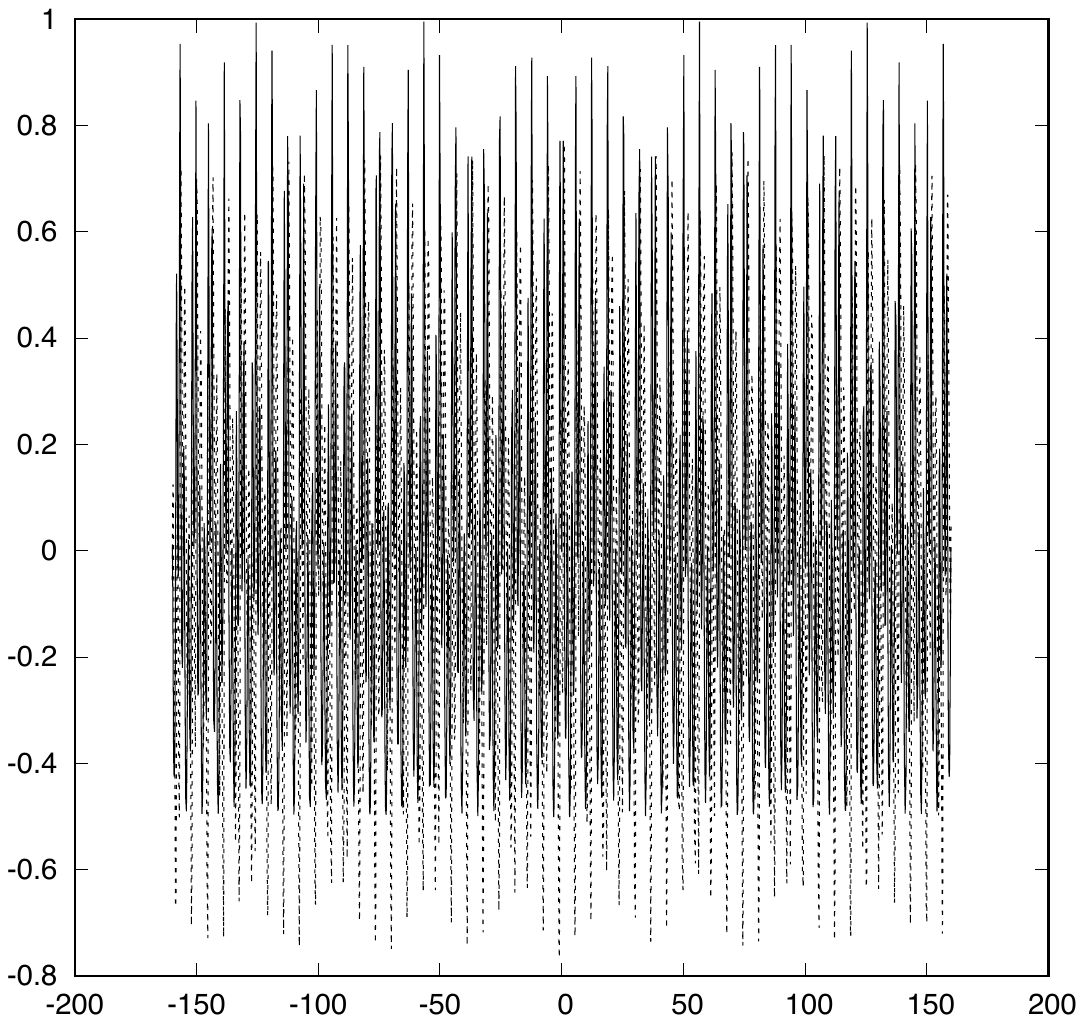}
        \caption{The initial data \newline $\phi_0 - (\sqrt{r} H_0(kr))'/(H_0(kr))$}
        \label{fig:sqr_dat}
    \end{subfigure}
    \caption{Numerical results for the potential $q(r)$ which is identically one on the interval $[1.5,2.5]$ and zero otherwise. Here $n=0.$ The time to generate the data was 65 seconds, and the time to solve was 130 seconds. The solve was done using frequencies in the range $[-160,160]$ and a spatial step size of $1/20 000.$}\label{fig:square_pot}
\end{figure}


\clearpage

\begin{thebibliography}{10}
\bibliographystyle{siam}
\bibitem{abram}
{\sc M.~Abramowitz and I.~A. Stegun}, eds., {\em {H}andbook of {M}athematical
  {F}unctions: with {F}ormulas, {G}raphs, and {M}athematical {T}ables},
  {N}ational {B}ureau of {S}tandards, 1964.

\bibitem{lgc4}
{\sc G.~Bao, S.~Hou, and P.~Li}, {\em Inverse scattering by a continuation
  method with initial guesses from a direct imaging algorithm}, J. Comput.
  Phys., 227 (2007), pp.~755--762.

\bibitem{lgc5}
{\sc G.~Bao and P.~Li}, {\em Inverse medium scattering for the helmholtz
  equation at fixed frequency}, Inverse Problems, 21 (2005), pp.~1621--1641.

\bibitem{lgc11}
{\sc G.~Bao and F.~Triki}, {\em Error estimates for the recursive linearization
  of inverse medium problems}, J. Comput. Math, 28 (2010), pp.~725--744.

\bibitem{lgc10}
\leavevmode\vrule height 2pt depth -1.6pt width 23pt, {\em Inverse scattering
  problems with multi-frequencies}, Inverse Problems, 31 (2015).

\bibitem{les1}
{\sc C.~Borges, A.~Gillman, and L.~Greengard}, {\em High {R}esolution {I}nverse
  {S}cattering in {T}wo {D}imensions {U}sing {R}ecursive {L}inearization}, Siam
  J. Imaging Sci., 10 (2016), pp.~641--664.

\bibitem{gaussquad}
{\sc J.~Bremer, Z.~Gimbutas, and V.~Rokhlin}, {\em A {N}onlinear {O}ptimization
  {P}rocedure for {G}eneralized {G}aussian {Q}uadratures}, SIAM J. Sci.
  Comput., 32 (2010), pp.~1761--1788.

\bibitem{chad1}
{\sc K.~Chadan and P.~Sabatier}, {\em Inverse problems in quantum scattering
  theory}, Springer-Verlag, 1977.

\bibitem{yuchen}
{\sc Y.~Chen and V.~Rokhlin}, {\em On the {I}nverse {S}cattering {P}roblem for
  the {H}elmholtz {E}quation in {O}ne {D}imension}, Inverse Problems, 8 (1992).

\bibitem{codd}
{\sc E.~Coddington and N.~Levinson}, {\em Theory of ordinary differential
  equations}, McGraw Hill, New York, 1955.

\bibitem{lgc31}
{\sc D.~Colton and A.~Kirsch}, {\em An approximation problem in inverse
  scattering theory}, Appl. Anal., 41 (1991), pp.~23--32.

\bibitem{lgc35}
{\sc D.~Colton and P.~Monk}, {\em The inverse scattering problem for
  time-harmonic acoustic waves in an inhomogeneous medium}, Quart. J. Mech.
  Appl. Math., 41 (1988), pp.~97--125.

\bibitem{crutch1}
{\sc W.~Y. Crutchfield}, {\em Class of exact inversion solutions to vibrating
  string problems}, Physics {L}etters, 5 (1983), pp.~233--236.

\bibitem{deift1}
{\sc P.~Deift and E.~Trubowitz}, {\em Inverse scattering on the line}, Comm.
  Pure Appl. Math, 32 (1979), pp.~121--251.

\bibitem{fedor}
{\sc M.~V. Fedoryuk}, {\em Asymptotic {A}nalysis}, Springer-Verlag, 1993.

\bibitem{gelf}
{\sc I.~Gel'fand and B.~M. Levitan}, {\em On the determination of a
  differential equation by its spectral function}, Dokl. Akad. USSR, 77 (1951),
  pp.~557--560.

\bibitem{lgc44}
{\sc S.~Gutman and M.~Klibanov}, {\em Regularized quasi-newton method for
  inverse scattering problems}, Math. Comput. Modelling, 18 (1993), pp.~5--31.

\bibitem{lgc45}
\leavevmode\vrule height 2pt depth -1.6pt width 23pt, {\em Two versions of
  quasi-newton method for multi-dimensional inverse scattering problems at
  fixed frequencies}, J. Comput. Acoust., 1 (1993), pp.~197--228.

\bibitem{lgc46}
\leavevmode\vrule height 2pt depth -1.6pt width 23pt, {\em Iterative methods
  for multi-dimensional inverse scattering problems at fixed frequencies},
  Inverse Problems, 10 (1994), p.~573.

\bibitem{lgc49}
{\sc T.~Hohage}, {\em On the numerical solution of a three-dimensional inverse
  medium scattering problem}, Inverse Problems, 17 (2001), pp.~1743--1763.

\bibitem{lgc51}
{\sc M.~Ikehata}, {\em Reconstruction of an obstacle from the scattering
  amplitude at a fixed frequency}, Inverse Problems, 14 (1998), pp.~949--954.

\bibitem{endpoint}
{\sc S.~Kapur and V.~Rokhlin}, {\em High-order corrected trapezoidal quadrature
  rules for singular functions}, SIAM J. Numer. Anal., 34 (1997),
  pp.~1331--1356.

\bibitem{lgc58}
{\sc R.~E. Kleinman and P.~M. van~den Berg}, {\em A modified gradient method
  for two-dimensional problems in tomography}, J. Comput. Appl. Math., 42
  (1992), pp.~17--35.

\bibitem{lgc59}
\leavevmode\vrule height 2pt depth -1.6pt width 23pt, {\em An extended
  range-modified gradient technique for profile inversion}, Radio Sci., 28
  (1993), pp.~877--884.

\bibitem{lines1}
{\sc L.~Lines and S.~Treitel}, {\em A review of least-squares inversion and its
  applications to geophysical problems}, Geophysical {P}rospecting, 32 (1984),
  pp.~159--186.

\bibitem{perov}
{\sc D.~S. Mitrinovi{\'{c}}, J.~E. Pe{\v{c}}ari{\'{c}}, and A.~M. Fink}, {\em
  Inequalities {I}nvolving {F}unctions and {T}heir {I}ntegrals and
  {D}erivatives}, Springer {N}etherlands, 1991.

\bibitem{pan1}
{\sc G.~Pan and R.~A. Phinney}, {\em Full-waveform inversion of plane-wave
  seismogram in stratified acoustic media: {A}pplications and limitations},
  Geophysics, 54 (1989), pp.~568--580.

\bibitem{lgc68}
{\sc R.~Potthast}, {\em A point source method for inverse acoustic and
  electromagnetic obstacle scattering prob- lems}, IMA J. Appl. Math., 61
  (1998), pp.~119--140.

\bibitem{stark}
{\sc H.~Stark}, ed., {\em Image {R}ecovery: {T}heory and {A}pplication},
  Academic Press, Inc., 1987.

\bibitem{stick1}
{\sc D.~Stickler}, {\em Application of the trace formula methods to inverse
  scattering for some geophysical problems}, Inverse Problems (SIAM AMS
  Proceedings 14),  (1984), pp.~13--30.

\bibitem{sylv1}
{\sc J.~Sylvester}, {\em A {C}onvergent {L}ayer {S}tripping {A}lgorithm for the
  {R}adially {S}ymmetric {I}mpedance {T}omagraphy {P}roblem}, Communications in
  Partial Differential Equations, 17 (1992), pp.~1955--1994.

\bibitem{symes1}
{\sc W.~W. Symes and J.~J. Carazzone}, {\em Velocity inversion by differential
  semblance optimization}, Geophysics, 56 (1989), pp.~654--663.

\end{thebibliography}
\end{document}